\documentclass[11pt]{amsart}
\usepackage{float}
\usepackage[usenames]{color}
\usepackage{graphicx}
\usepackage{amssymb}
\usepackage{amscd}
\theoremstyle{plain}
\newtheorem{thm}{Theorem}[section]
\newtheorem{lemma}[thm]{Lemma}

\newtheorem{conj}[thm]{Conjecture}
\newtheorem{prop}[thm]{Proposition}

\theoremstyle{definition}

\newtheorem{example}[thm]{Example}

\def\ker{\mathop{\hbox{Ker}}\nolimits}

\newcommand{\fra}{\mathfrak{a}}
\newcommand{\frb}{\mathfrak{b}}

\newcommand{\frg}{\mathfrak{g}}
\newcommand{\frh}{\mathfrak{h}}

\newcommand{\frk}{\mathfrak{k}}
\newcommand{\frl}{\mathfrak{l}}

\newcommand{\frp}{\mathfrak{p}}
\newcommand{\frq}{\mathfrak{q}}

\newcommand{\frt}{\mathfrak{t}}
\newcommand{\fru}{\mathfrak{u}}

\newcommand{\bbC}{\mathbb{C}}

\newcommand{\bbN}{\mathbb{N}}

\newcommand{\bbR}{\mathbb{R}}

\newcommand{\bbZ}{\mathbb{Z}}

\newcommand{\caL}{\mathcal{L}}

\newcommand{\caR}{\mathcal{R}}

\newcommand{\be}{\begin {equation}}
\newcommand{\ee}{\end {equation}}
\newcommand{\bp}{\begin {proof}}
\newcommand{\ep}{\end {proof}}

\begin{document}

\title[A finiteness result for Dirac cohomology]
{Unitary Representations with Dirac cohomology:\\ a finiteness result for complex Lie groups}

\author{Jian Ding}
\address[Ding]{School of Mathematics, Hunan University, Changsha 410082,
P.~R.~China}
\email{dingjain@hnu.edu.cn}

\author{Chao-Ping Dong}
\address[Dong]{School of Mathematics, Hunan University, Changsha 410082,
P.~R.~China}
\email{chaopindong@163.com}
\thanks{Dong is supported by NSFC grant 11571097 and the China Scholarship Council.}

\abstract{Let $G$ be a connected complex simple Lie group, and let $\widehat{G}^{\mathrm{d}}$ be the set of all equivalence classes of irreducible unitary representations with non-vanishing Dirac cohomology. We show that $\widehat{G}^{\mathrm{d}}$ consists of two parts: finitely many scattered representations, and finitely many strings of representations.  Moreover,  the strings of $\widehat{G}^{\mathrm{d}}$ come from $\widehat{L}^{\mathrm{d}}$ via cohomological induction and they are all in the good range. Here $L$ runs over the Levi factors of proper $\theta$-stable  parabolic subgroups of $G$. It follows that figuring out $\widehat{G}^{\mathrm{d}}$ requires a finite calculation in total. As an application, we report a complete description of $\widehat{F}_4^{\mathrm{d}}$.}
\endabstract



\maketitle

\section{Introduction}

In 1928, by using matrix algebra, Dirac introduced the eponymous operator   in his description of the wave function of the spin-$\frac{1}{2}$ massive particles such as electrons and quarks \cite{Di}. This operator was a square root of the wave operator, and it led to the foundational Dirac equation in quantum mechanics.

The influence of Dirac equation was not restricted within physics. For instance, based on earlier works by Harish-Chandra \cite{HC1, HC2},
Parthasarathy introduced the Dirac operator for semisimple Lie groups to give a geometric construction for most of the discrete series representations \cite{P1}  in 1972.  A byproduct is Parthasarathy's Dirac inequality \eqref{Dirac-inequality}, which effectively detects non-unitarity. This project was completed by Atiyah and Schmid: they showed that all the discrete series can be found in the kernel of the Dirac operator \cite{AS}.

To sharpen the Dirac inequality, and to understand the unitary dual better, Vogan introduced Dirac cohomology in 1997. See \eqref{def-Dirac-cohomology}. It was obvious from the definition that Dirac cohomology is an invariant for Lie group representations.
A subsequent interesting problem was to classify all the irreducible unitary representations with non-vanishing Dirac cohomology.
Huang and Pand\v zi\'c \cite{HP} proved the Vogan conjecture in 2002. Their work is foundational for the computation of Dirac cohomology. The first aim of this paper is to report a finiteness theorem for the classification for \emph{complex} Lie groups. Inspired by Salamanca-Riba \cite{Sa}, Huang, Kang and Pand\v zi\'c \cite{HKP}, we will adopt the cohomological induction approach.

Now let us be more precise. Let $G$ be a connected complex simple Lie group. We view $G$ as a real Lie group, and let $\theta$ be the Cartan involution of $G$.
Then $K:=G^{\theta}$ is a maximal compact subgroup of $G$.
Write the Lie algebra of $G$ (resp. $K$) as $\frg_0$ (resp. $\frk_0$). We will drop the subscripts to denote the complexifications. Let $\frg_0=\frk_0+\frp_0$ be the Cartan decomposition.  Fix
an orthonormal basis $Z_1,\dots, Z_n$ of $\frp_0$ with respect to
the inner product induced by the Killing form $\langle\, ,\,\rangle$. Let $U(\frg)$ be the
universal enveloping algebra of $\frg$ and let $C(\frp)$ be the
Clifford algebra of $\frp$ with respect to $\langle\, ,\,\rangle$. The Dirac operator
$D\in U(\frg)\otimes C(\frp)$ is defined as
$$D=\sum_{i=1}^{n}\, Z_i \otimes Z_i.$$
It is easy to check that $D$ does not depend on the choice of the
orthonormal basis $\{Z_i\}_{i=1}^{n}$, and it is $K$-invariant for the diagonal
action of $K$ given by adjoint actions on both factors. Let $\Delta: \frk \to U(\frg)\otimes C(\frp)$ be given by $\Delta(X)=X\otimes 1 + 1\otimes \alpha(X)$, where $\alpha$ is the action map $\frk\to \mathfrak{s}\mathfrak{o}(\frp)$ followed by the usual identifications $\mathfrak{s}\mathfrak{o}(\frp)\cong\wedge^2(\frp)\hookrightarrow C(\frp)$. We denote the image of $\frk$ by $\frk_{\Delta}$, and denote by $\Omega_{\frg}$ (resp. $\Omega_{\frk}$)
the Casimir operator of $\frg$ (resp. $\frk$). Let $\Omega_{\frk_{\Delta}}$ be the image of $\Omega_{\frk}$ under $\Delta$. Then as was firstly obtained by Parthasarathy \cite{P1}, we have
\begin{equation}\label{D-square}
D^2=-\Omega_{\frg}\otimes 1 + \Omega_{\frk_{\Delta}} + (\|\rho_c\|^2-\|\rho\|^2) 1\otimes 1.
\end{equation}

Let $\widetilde{K}$ be the spin double cover of $K$, which is a subgroup of $K\times \text{Spin}\,\frp_0$. Let $\pi$ be a
($\frg$, $K$)-module, and let $S_G$ be a spin module for
$C(\frp)$. Then $\pi\otimes S_G$ is a $(U(\frg)\otimes C(\frp),
\widetilde{K})$ module. Indeed, $U(\frg)\otimes C(\frp)$ acts on $\pi\otimes S_G$ in
the obvious way, while $\widetilde{K}$ acts on $\pi$
through $K$ and on $S_G$ through the spin group
$\text{Spin}\,{\frp_0}$.
In particular, the Dirac operator $D$ acts on $\pi\otimes S_G$, and the Dirac cohomology of a $(\frg, K)$-module $\pi$ is defined as the $\widetilde{K}$-module
\begin{equation}\label{def-Dirac-cohomology}
H_D(\pi)=\text{Ker}\, D/ (\text{Im} \, D \cap \text{Ker} D).
\end{equation}

We care the most about the case when $\pi$ is unitary. Then $D$ is self-adjoint with respect to a natural inner product on $\pi\otimes S_G$, and we have
\begin{equation}\label{Dirac-unitary}
H_D(\pi)=\text{Ker}\, D=\text{Ker}\, D^2.
\end{equation}
Note that $D^2$ has non-negative eigenvalue on any $\widetilde{K}$-type of $\pi\otimes S_G$. Utilizing this and \eqref{D-square}, one  deduces Parthasarathy's Dirac operator inequality: \begin{equation}\label{Dirac-inequality}
\|\gamma+\rho_c\|\geq \|\Lambda\|,
\end{equation}
where $\gamma$ is the highest weight of any $\widetilde{K}$-type of $\pi\otimes S_G$, and $\Lambda$ is the infinitesimal character of $\pi$. Moreover, by Theorem \ref{thm-HP}, it becomes equality on some $\widetilde{K}$-types of $\pi\otimes S_G$ if and only if $H_D(\pi)$ is non-vanishing (see Proposition \ref{prop-D-spin-lowest} for more).

Let $\widehat{G}^{\mathrm{d}}$ be the set of all equivalence classes of irreducible unitary representations of $G$ with nonzero Dirac cohomology.
The first result of the current paper is the following description of $\widehat{G}^{\mathrm{d}}$.

\medskip
\noindent\textbf{Theorem A.}
\emph{The set $\widehat{G}^{\mathrm{d}}$ for a connected complex simple Lie group consists of two parts:
\begin{itemize}
\item[a)] finitely many scattered modules (the scattered part); and
\item[b)] finitely many strings of modules (the string part).
\end{itemize}
Here the scattered part consists exactly of the unitary modules $J(\lambda, -s\lambda)$ in \eqref{BP} such that each simple reflection occurs in any reduced expression of the involution $s$, and that $H_D(J(\lambda, -s\lambda))\neq 0$.
Moreover,  modules in the string part of $G$ are all cohomologically induced from the scattered part of $\widehat{L}_{\mathrm{ss}}^{\mathrm{d}}$ tensored with one-dimensional unitary characters of $L$, and they are all in the good range. Here $L$ runs over the Levi subgroups of proper $\theta$-stable parabolic subgroups of $G$,  and $L_{\mathrm{ss}}$ denotes the semisimple factor of $L$.}
\medskip

The proof involves the following ingredients:
an analysis of Parthasarathy's Dirac inequality, results from Vogan \cite{Vog84}, and Theorem 6.1 of \cite{D}. The key idea is arranging the candidate representations into $s$-families, see Section \ref{sec-s-family}.
In view of Theorem A, to figure out $\widehat{G}^{\mathrm{d}}$, it suffices to understand the scattered parts of $\widehat{G}^{\mathrm{d}}$ and $\widehat{L}_{\mathrm{ss}}^{\mathrm{d}}$ for the finitely many Levis.
Furthermore, Proposition \ref{prop-I(s)-empty} says that it boils down to considering finitely many candidate representations to sieve out the scattered part of $\widehat{G}^{\mathrm{d}}$.
Therefore, pinning down $\widehat{G}^{\mathrm{d}}$ requires a finite calculation in total. From this aspect, we interpret Theorem A as a \emph{finiteness} result.

We introduce a computing method that allows us to sieve out the finitely many candidate representations more efficiently. The basic idea is to study the distribution of the spin norm along the Vogan pencil \cite{Vog81} starting from the lowest $K$-type, see Section \ref{sec-comp} for more. Thanks to the breakthrough in achieving an algorithm for computing unitarity by Adams, van Leeuwen, Trapa and Vogan \cite{ALTV}, and thanks to the recent development of the software \texttt{atlas} \cite{At}, one can eventually handle these finitely many candidates completely in low rank cases.

As an application, let us report the following complete description of $\widehat{F}_4^{\mathrm{d}}$.

\medskip
\noindent\textbf{Theorem B.}
\emph{The set $\widehat{F}_4^{\mathrm{d}}$ for the complex $F_4$ consists of ten scattered representations (see Table \ref{table-F4-scattered-part}) whose spin-lowest $K$-types are all unitarily small, and thirty strings of representations (see Table \ref{table-F4-string-part}).
Moreover, each representation $\pi\in\widehat{F}_4^{\mathrm{d}}$ has a unique spin-lowest $K$-type, and this $K$-type occurs exactly once.}
\medskip

In Theorem B, the notion unitarily small (\emph{u-small} for short) $K$-type was introduced by Salamanca-Riba and Vogan \cite{SV} in their unified conjecture on the shape of $\widehat{G}$, the unitary dual of $G$. The notion spin-lowest $K$-type will be recalled in Section \ref{sec-u-small-spin-norm}.


Finally, we make the following.

\medskip
\noindent\textbf{Conjecture C.}
\emph{Let $G$ be a connected complex simple Lie group. Any $\pi$ in the scattered part of $\widehat{G}^{\mathrm{d}}$ has a unique spin-lowest $K$-type which must be u-small.}
\medskip

If Conjecture C holds, then by Theorem A, we could conclude that any $\pi\in\widehat{G}^{\mathrm{d}}$ has a unique spin-lowest $K$-type.
This phenomenon does not hold for real Lie groups, see Barbasch and Pand\v zi\'c \cite{BP2}.

The paper is organized as follows.
We set up the notation and collect necessary preliminaries in Section 2. Theorem A is proved in Section 3. Then we study the Dirac cohomology of tempered representations, minimal representations and model representations in Section 4. Dirac cohomology of the spherical unitary dual is investigated in Section 5. Then we introduce a computing method and illustrate it carefully for the $G_2$ case in Section 6. Sections 7 and 8 are devoted to determining the string part and the scattered part of $\widehat{F}_4^{\mathrm{d}}$, respectively. Finally, Section 9 is an appendix indexing the $140$ involutions of $F_4$.

Throughout this paper $\bbN=\{0, 1, 2, \dots\}$, $\mathbb{P}=\{1, 2, \dots\}$
and $\frac{1}{2}\mathbb{P}=\{\frac{1}{2}, 1, \frac{3}{2}, 2, \dots\}$.

\medskip
\noindent\textbf{Added Notes.} The mathematical content of the current paper has been kept as that of https://arxiv.org/abs/1702.01876. Over the past three years, this work has inspired further progresses towards the classification of $\widehat{G}^d$. Let us mention the following ones.

\begin{itemize}
\item[$\bullet$] The computing method in Section \ref{sec-comp} has been improved in \cite{D17}. As a consequence, the set $\widehat{E_6}^d$ has been pinned down there.

\item[$\bullet$] Theorem A has been extended to real reductive Lie groups later by the second author in \cite{D17r}. However, the current approach starts from the very effective reduction \eqref{BP} carried out by Barbasch and Pand\v zi\'c \cite{BP}, while the analogue of \eqref{BP} for real reductive Lie groups is unclear yet. Thus the approach adopted in \cite{D17r} actually \emph{differs} from here. Moreover, the current language is traditional, while the one in \cite{D17r} is that of \texttt{atlas} \cite{ALTV, At}.

\item[$\bullet$] A classification of $\widehat{G}^d$ for complex classical Lie groups will be reported in \cite{BDW}. As consequences, all the relevant conjectures in \cite{BP} will be answered in the affirmative for complex classical groups.

\item[$\bullet$] An understanding  of the scattered representations of $SL(n, \bbC)$ has been given in \cite{DW}. In particular, Conjecture 5.2 of \cite{D17},  Conjecture 5.6(a) of the current paper will be confirmed.
\end{itemize}

A final remark is that in view of the recent research announcement \cite{BP3}, the classification of $\widehat{G}^d$ would hopefully have applications in automorphic forms.

\emph{Acknowledgements.}
Dong thanks the math department of MIT for offering excellent working conditions during October 2016 and September 2017. He is deeply grateful to the \texttt{atlas} mathematicians for many things, and to the referee for offering nice suggestions.


\section{Preliminaries}
Although some results in this section (say Theorems \ref{thm-HP} and \ref{thm-Vogan}, Proposition \ref{prop-D-spin-lowest}) hold for real reductive Lie groups, for simplicity, we only quote them under the assumption that $G$ is a connected complex simple Lie group.

We continue with the notation in the introduction. Let $T$ be a maximal torus of $K$. Let $\fra_0=\sqrt{-1}\frt_0$
and $A=\exp(\fra_0)$. Then up to conjugation, $H=TA$ is the unique
$\theta$-stable Cartan subgroup of $G$.
We identify
\begin{equation}\label{identifications}
\frg\cong \frg_{0} \oplus
\frg_0, \quad
\frh\cong \frh_{0} \oplus
\frh_0,\quad \frt\cong \{(x,-x) : x\in
\frh_{0} \}, \quad\fra \cong\{(x, x) : x\in
\frh_{0} \}.
\end{equation}

Fix a Borel subgroup $B$ of $G$
containing $H$. Put $\Delta^{+}(\frg_0, \frh_0)=\Delta(\frb_0, \frh_0)$.
Then we have the corresponding simple roots $\alpha_1, \dots, \alpha_l$ and fundamental weights $\varpi_1, \dots, \varpi_l$. Note that
$$
\langle \varpi_i, \check{\alpha_j}\rangle=\delta_{ij},
$$
where $\delta_{ij}$ is the Kronecker notation.
 Set $[l]:=\{1,2, \dots, l\}$.  Denote by $s_i$ the simple reflection $s_{\alpha_i}$.
Let $\rho$ be the half sum of positive roots in $\Delta^{+}(\frg_0, \frh_0)$. In this paper, we always use the fundamental weights as a basis to express a weight. That is, $[n_1, \dots, n_l]$ stands for the weight $\sum_{i=1}^{l} n_i \varpi_i$. For instance, $\rho=[1,1,1,1]$ for complex $F_4$. Set
$$
\Delta^+(\frg, \frh)=\Delta^+(\frg_0, \frh_0) \times \{0\} \cup \{0\} \times (-\Delta^+(\frg_0, \frh_0)),
$$
which is $\theta$-stable. When restricted to $\frt$, we get $\Delta^+(\frg, \frt)$, $\Delta^+(\frk, \frt)$ and $\Delta^+(\frp, \frt)$. Denote by $\rho_\frg$ (resp., $\rho_c$) the half-sum of roots in $\Delta^+(\frg, \frh)$ (resp., $\Delta^+(\frk, \frt)$).
Note that we can identify $\rho_\frg=(\rho, -\rho)$ with $2\rho\in\frh_{0}^{*}$ via \eqref{identifications}. Similarly, $\rho_c$ can be identified with $\rho\in\frh_{0}^{*}$ via \eqref{identifications}. We denote by $W$ the Weyl group
$W(\frg_0, \frh_0)$, which has identity element $e$ and longest element $w_0$. Then $W(\frg, \frh)\simeq W \times W$.

\subsection{Zhelobenko classification}
Let $(\lambda_{L}, \lambda_{R})\in \frh_0^{*}\times
\frh_0^{*}$ be such that $\lambda_{L}-\lambda_{R}$ is
a weight of a finite dimensional holomorphic representation of $G$.
Using \eqref{identifications}, we can view $(\lambda_L, \lambda_R)$ as a real-linear functional on $\frh$, and write $\bbC_{(\lambda_L, \lambda_R)}$ as the character of $H$ with differential $(\lambda_L, \lambda_R)$ (which exists). Using \eqref{identifications} again, we have
$$
\bbC_{(\lambda_L, \lambda_R)}|_{T}=\bbC_{\lambda_L-\lambda_R}, \quad \bbC_{(\lambda_L, \lambda_R)}|_{A}=\bbC_{\lambda_L+\lambda_R}.
$$
Extend $\bbC_{(\lambda_L, \lambda_R)}$ to a character of $B$, and put $$I(\lambda_{L}, \lambda_{R})
:={\rm Ind}_{B}^{G}
(
\bbC_{(\lambda_L, \lambda_R)}
)_{K-finite}
.$$

Let $V_{\delta}$ be the $K$-type with highest weight $\delta$. For convenience, we may simply refer to $V_{\delta}$ as $\delta$. We will treat $\frk$-types and $\widetilde{K}$-types similarly. Given an arbitrary weight $\mu\in\frt^*$, let $\{\mu\}$ be the unique dominant weight to which $\mu$ is conjugate under the action of $W(\frk, \frt)$.

\begin{thm}\label{thm-Zh} {\rm (Zhelobenko \cite{Zh})}
The $K$-type $V_{\{\lambda_{L}-\lambda_{R}\}}$
occurs with multiplicity one in
$I(\lambda_{L}, \lambda_{R})$. Let
$J(\lambda_L,\lambda_R)$ be the unique subquotient of
$I(\lambda_{L}, \lambda_{R})$ containing $V_{\{\lambda_{L}-\lambda_{R}\}}$.
\begin{itemize}
\item[a)] Every irreducible admissible ($\frg$, $K$)-module is of the form $J(\lambda_L,\lambda_R)$.
\item[b)] Two such modules $J(\lambda_L,\lambda_R)$ and
$J(\lambda_L^{\prime},\lambda_R^{\prime})$ are equivalent if and
only if there exists $w\in W$ such that
$w\lambda_L=\lambda_L^{\prime}$ and $w\lambda_R=\lambda_R^{\prime}$.
\item[c)] $J(\lambda_L, \lambda_R)$ admits a nondegenerate Hermitian form if and only if there exists
$w\in W$ such that $w(\lambda_L-\lambda_R) =\lambda_L-\lambda_R , w(\lambda_L+\lambda_R) = -\overline{(\lambda_L+\lambda_R)}$.
\item[d)] The representation $I(\lambda_{L}, \lambda_{R})$ is tempered if and only if $\lambda_{L}+\lambda_{R}\in i\frh_0^*$. In this case,
$I(\lambda_{L}, \lambda_{R})=J(\lambda_{L}, \lambda_{R})$.
\end{itemize}
\end{thm}

Up to equivalence,   $J(\lambda_L, \lambda_R)$ is the unique irreducible admissible $(\frg, K)$-module with lowest $K$-type $\{\lambda_{L}-\lambda_{R}\}$ and  infinitesimal character  the $W\times W$ orbit of $(\lambda_L, \lambda_R)$.
We call the pair $(\lambda_{L}, \lambda_{R})$  \emph{Zhelobenko parameters} of $J(\lambda_{L}, \lambda_{R})$. For convenience, we will also call $\lambda_L-\lambda_R$ (resp. $\lambda_L+\lambda_R$) the $T$-parameter (resp. $A$-parameter) of $J(\lambda_{L}, \lambda_{R})$.

\subsection{Dirac cohomology}

We embed $\frt^{*}$ as a subspace of $\frh^{*}$ by setting  the linear functionals on $\frt$ to be zero on $\fra$. Now let us state the proven Vogan conjecture.

\begin{thm}{\rm (Huang and Pand\v zi\'c \cite{HP})}\label{thm-HP}
Let $\pi$ be an irreducible ($\frg$, $K$)-module.
Assume that the Dirac
cohomology of $\pi$ is nonzero, and that it contains the $\widetilde{K}$-type $E_{\gamma}$ with highest weight $\gamma\in\frt^{*}\subset\frh^{*}$. Then the infinitesimal character of $\pi$ is conjugate to
$\gamma+\rho_{c}$ under $W(\frg,\frh)$.
\end{thm}

\subsection{Cohomological induction}
For complex Lie groups, cohomological induction is essentially equivalent to the ordinary parabolic induction. However, adopting the former setting will give us convenience in utilizing existing results on cohomological induction. Let us give a brief review. Fix a nonzero element $H\in i\frt_0$, then  a $\theta$-stable parabolic subalgebra $\frq= \frl\oplus\fru$ of $\frg$ can be defined as the sum of nonnegative eigenspaces of $\mathrm{ad}(H)$. Here the Levi subalgebra $\frl$ of
$\frq$ is the zero eigenspace of $\mathrm{ad}(H)$, while the
nilradical $\fru$ of $\frq$ is the sum of positive eigenspaces of
$\mathrm{ad}(H)$.  Then it follows from $\theta(H)=H$ that $\frl$, $\fru$ and $\frq$ are all
$\theta$-stable. Let $L=N_G(\frq)$, which is connected and has $K_L:=L\cap K$ as a maximal compact subgroup.

Let us arrange the positive root systems in a compatible way, that
is, $\Delta(\fru)\subseteq \Delta^{+}(\frg,\frh)$ and set
$\Delta^{+}(\frl, \frh)=\Delta(\frl, \frh)\cap
\Delta^{+}(\frg,\frh)$.  We denote by $\rho(\fru)$  the half sum of roots in
$\Delta(\fru,\frh)$.

Let $Z$ be an ($\frl$, $K_L$) module. Cohomological induction functors attach to $Z$ certain ($\frg, K$)-modules $\caL_j(Z)$ and $\caR^j(Z)$, where $j$ is a nonnegative integer. For a definition, see Chapter 2 of \cite{KV}.
Suppose that $\lambda$ is the infinitesimal character of $Z$. We say
$Z$ or $\lambda$ is {\it in good range} if
\begin{equation}\label{def-good}
\mathrm{Re}\langle \lambda +\rho(\fru),\, \alpha \rangle >
0, \quad \forall \alpha\in \Delta(\fru, \frh).
 \end{equation}

Now we are able to state the results pertaining to cohomological induction that we shall need in this paper. Note that the first theorem was mainly due to Vogan \cite{Vog84}, while the second one was obtained in the second named author's thesis, see also Theorem 6.1 of \cite{D}.

\begin{thm}\label{thm-Vogan}
{\rm (Theorems 1.2 and 1.3 of \cite{Vog84}, or  Theorems 0.50 and 0.51 of \cite{KV})}
Suppose the admissible
 ($\frl$, $L\cap K$)-module $Z$ is in the good range.  Then we have
\begin{itemize}
\item[a)] $\caL_j(Z)=\caR^j(Z)=0$ for $j\neq S$, where $S:=\emph{\text{dim}}\,(\fru\cap\frk)$.
\item[b)] $\caL_S(Z)\cong\caR^S(Z)$ as ($\frg$, $K$)-modules. They are both nonzero.
\item[c)]  $\caL_S(Z)$  is irreducible if and only if $Z$ is irreducible.
\item[d)]
$\caL_S(Z)$  is unitary if and only if $Z$ is unitary.
\end{itemize}
\end{thm}

\begin{thm}\label{thm-D}
Let $G$ be a connected complex simple Lie group.
Let $Z$ be an irreducible unitary
($\frl$, $K_L$) module with infinitesimal character $\lambda\in
i\frt_0^{*}$ which is dominant for
$\Delta^{+}(\frl\cap\frk, \frt)$. Assume that
\begin{equation}\label{complex-cond}
\lambda+\rho(\fru) \mbox{ is dominant integral regular for }
\Delta^{+}(\frk,\frt).
\end{equation}
Then $H_D(\caL_S(Z))$ is nonzero if and only if $H_D(Z)$ is
nonzero.
\end{thm}

\subsection{Unitarily small $K$-types, spin norm and spin lowest $K$-type}\label{sec-u-small-spin-norm}
By Theorem 6.7 of \cite{SV}, the $K$-type $\delta$ is u-small if and only if there is an expression
$$
\delta=\sum_{\beta\in \Delta(\frp, \frt)} b_{\beta}\beta \qquad (0\leq b_{\beta} \leq 1).
$$
Equivalently, the $K$-type $\delta$ is u-small if and only if $\langle \delta-2\rho, \varpi_i\rangle\leq 0$, $1\leq i\leq l$.

Given a general $K$-type $\delta$, its \emph{spin norm} is defined as
\begin{equation}\label{Spin-norm-K-type}
\|\delta\|_{\mathrm{spin}} := \| \{\delta-\rho\} + \rho \|.
\end{equation}
For any
irreducible admissible ($\frg$, $K$)-module $\pi$, we define
$\|\pi\|_{\mathrm{spin}}$ as the minimum of the spin norm of all its $K$-types.
We call $\delta$ a \emph{spin lowest $K$-type} of $\pi$ if
it occurs in $\pi$ and
$\|\delta\|_{\mathrm{spin}}=\|\pi\|_{\mathrm{spin}}$.

The following result is taken from the second named author's thesis. It is a combination of the ideas and results of
Parthasarathy \cite{P1, P2}, Vogan \cite{Vog97}, Huang and Pand\v
zi\'c (see Theorem 3.5.2 of \cite{HP2}). It suggests that spin norm and spin lowest $K$-type give the right framework for the classification of $\widehat{G}^{\mathrm{d}}$.

\begin{prop}\label{prop-D-spin-lowest}
For any irreducible unitary ($\frg$, $K$)-module $\pi$ with
infinitesimal character $\Lambda$, let $\delta$ be any $K$-type
occurring in $\pi$. Then
\begin{enumerate}
\item[a)] $\|\pi\|_{\mathrm{spin}}\geq\|\Lambda\|$, and the equality holds if and only if $H_D(\pi)$ is nonzero.
\item[b)] $\|\delta\|_{\mathrm{spin}}\geq \|\Lambda\|$, and the equality holds
if and only if $\delta$ contributes to $H_D(\pi)$.
\item[c)] If $H_D(\pi)\neq 0$, it is exactly the spin lowest $K$-types of $\pi$
that contribute to $H_D(\pi)$.
\end{enumerate}
\end{prop}
\begin{proof} As mentioned in \eqref{Dirac-unitary}, since $\pi$ is unitary, we have $H_D(\pi)=\ker D^2$. Moreover, by \eqref{D-square}, $D^2$ acts by the non-negative scalar
$$
\|\gamma+\rho_c\|^2-\|\Lambda\|^2
$$
on the $\widetilde{K}$-type $\gamma$ of $\pi\otimes S_G$. By Lemma 2.2 of \cite{BP}, which is a special case of Chapter II Lemma 6.9 of \cite{BW}, the spin module $S_G$ is a multiple of the $\frk$-type $\rho_c$. Moreover, $\|\gamma+\rho_c\|^2$ attains its minimum when $\gamma$ is the PRV component of $\delta\otimes \rho_c$, i.e., when $\gamma=\{\delta-\rho_c\}$ (see Corollaries
1 and 2 of Theorem 2.1 in \cite{PRV}). Since $\rho_c\in\frt^*$ can be identified with $\rho\in\frh_0^*$ via \eqref{identifications}, the proof finishes once we recall the spin norm of $\delta$ defined in \eqref{Spin-norm-K-type}.
\end{proof}

\subsection{Distribution of the spin norm along Vogan pencils}

Let $\beta$ be the highest root.
 We call a set of $K$-types
\begin{equation}\label{P-delta}
P(\delta):=\{\delta +n \beta \,|\, n\in\bbN \}
\end{equation}
a \emph{Vogan pencil}. For instance, $P(0)$ denotes the pencil starting from the trivial $K$-type. By Lemma 3.4 and Corollary 3.5 of \cite{Vog81}, the $K$-types occurring in an infinite-dimensional irreducible ($\frg$, $K$)-module $\pi$ consist of certain Vogan pencils.

Put
\begin{equation}\label{P-mu-prime}
P_{\delta}:=\min \{\|\delta +n \beta \|_{\mathrm{spin}}\,|\, n\in\bbN \}.
\end{equation}
Calculating $P_\delta$ will be vital for us in later sections.
Phrased in another way, Theorem 1.1 of \cite{D2} says that
\begin{equation}\label{P-mu}
P_\delta=
\begin{cases}
\min \{\|\delta +n \beta \|_{\mathrm{spin}}\,|\, \delta+n\beta \mbox{ is u-small}\} & \mbox{ if $\delta$ is u-small};\\
\|\delta\|_{\mathrm{spin}} & \mbox{ otherwise.}
\end{cases}
\end{equation}

\section{Proof of Theorem A}
This section aims to prove Theorem A.

\subsection{$s$-families}\label{sec-s-family}
As deduced by Barbasch and Pand\v zi\'c \cite{BP} from
Theorems \ref{thm-Zh} and \ref{thm-HP}, to find all the irreducible unitary representations with nonzero Dirac cohomology,  it suffices to consider the following candidates
 \begin{equation}\label{BP}
J(\lambda, -s\lambda),
\end{equation}
where $s\in W$ is an \emph{involution}, and $2\lambda$ is dominant integral and regular.

For convenience of reader, we repeat part of the explanation from \cite{BP} that the  element $s$ in \eqref{BP} must be an involution. Indeed, for $J(\lambda, -s\lambda)$ to be unitary, it should admit a non-degenerate Hermitian form. Thus by Theorem \ref{thm-Zh}(c), there exists $w\in W$ such that
$$
w(\lambda+s\lambda)=\lambda+s\lambda, \quad w(\lambda-s\lambda)=-\lambda+s\lambda.
$$
Therefore, $w\lambda=s\lambda$ and $ws \lambda=\lambda$. Since $\lambda$ is regular, we must have $w=s$ and $ws=e$. Thus $s^2=e$, i.e., $s$ is an involution.

There are two ways of indexing the representations in \eqref{BP}. On one hand, we can fix $\lambda$, and let $s$ vary. For instance, Barbasch and Pand\v zi\'c fixed $\lambda=\frac{\rho}{2}$ and studied the representations $J(\frac{\rho}{2}, -s\frac{\rho}{2})$ carefully in Section 3 of \cite{BP}. These representations deserve particular attention since they have the smallest possible infinitesimal character.

On the other hand,  one can fix $s$ and let $\lambda$ varies.
Thinking in this way leads us to denote
\begin{equation}\label{Lambda-s}
 \Lambda(s):=\left\{\lambda=[\lambda_1, \dots, \lambda_l]\mid 2\lambda_i\in\mathbb{P} \mbox{ and } \lambda+s\lambda \mbox{ is integral}\right\}.
\end{equation}
We call $\Lambda(s)$ and the corresponding representations $J(\lambda, -s\lambda)$ an \emph{$s$-family}. Note that an $s$-family has infinitely many members. For instance, the $e$-family consists of tempered representations, and they will be handled in Section \ref{sec-several-families}; while on the other extreme, spherical representations live in the $w_0$-family, and they will be considered in Section \ref{sec-spherical}.

\subsection{Involutions and strings}

Fix an involution $s$. Put
\begin{equation}\label{I-s}
 I(s):=\left\{i\in[l]\mid s(\varpi_i)=\varpi_i\right\}.
\end{equation}
It turns out that this set will play an important role in subsequent discussions. Thus let us give an equivalent description for it. In particular, it says that the involution $s$ actually lives in the subgroup $\langle s_j\mid j\in[l]\setminus I(s)\rangle$ of $W$.

\begin{lemma}\label{lemma-s}
Let $s$ be an element of the Weyl group $W$. Then the following are equivalent.
\begin{itemize}
\item[a)] $s(\varpi_i)=\varpi_i$;
\item[b)] the simple reflection $s_i$ does not occur in some reduced expression of $s$;
\item[c)] the simple reflection $s_i$ does not occur in any reduced expression of $s$.
\end{itemize}
\end{lemma}
\begin{proof}
The only non-trivial step is to show
that (a) implies (c). Let $s=s_{\gamma_{1}}s_{\gamma_{2}}\cdots s_{\gamma_{n}}
$ be any reduced decomposition of $s$ into simple root reflections, and suppose that $j$ is the smallest index such that $\gamma_j=\alpha_i$. Then, by Lemma 5.5 of \cite{DH11},
$$\varpi_i-
s(\varpi_i)=\sum_{k=1}^{n} \langle
\varpi_i, \check{\gamma_{k}}\rangle\,
s_{\gamma_1}s_{\gamma_2}\cdots s_{\gamma_{k-1}}(\gamma_{k}),$$
where $s_{\gamma_1}s_{\gamma_2}\cdots s_{\gamma_{k-1}}(\gamma_{k})$
is a positive root for each $k$. In particular, when $k=j$, the term on the RHS is nonzero. Thus $s(\varpi_i)\neq\varpi_i$, contradiction.
\end{proof}

Let $I$ be a \emph{non-empty} subset of $[l]$. Recall the $s$-family $\Lambda(s)$ from \eqref{Lambda-s}. We call
\begin{equation}\label{I-string}
 \{\lambda\in\Lambda(s)\mid \lambda_i \mbox{ varies for } i\in I \mbox{ and } \lambda_j \mbox{ is fixed for } j\in[l]\setminus I\}
\end{equation}
and the corresponding representations $J(\lambda, -s\lambda)$ an \emph{$(s, I)$-string}.
When $s$ is known from the context, we may call it an $I$-string or just a string. Since the non-emptyness of $I$ is built into the definition, an $(s, I)$-string also contains infinitely many members.

\subsection{An analysis of Parthasatharathy's Dirac inequality}
Fix an involution $s$, and let $\lambda=\sum_{i=1}^l \lambda_i\varpi_i\in\Lambda(s)$.
Put $\mu=\{\lambda+s\lambda\}$. Then
\begin{align*}
\Delta_1(\lambda) :&=\|2\lambda\|^2-\|\mu\|_{\mathrm{spin}}^2\\
&=\|\lambda-s\lambda\|^2+\|\lambda+s\lambda\|^2 -\|\mu\|_{\mathrm{spin}}^2\\
&=\|\lambda-s\lambda\|^2+\|\mu\|^2 -\|\mu\|_{\mathrm{spin}}^2\\
&=\|\lambda-s\lambda\|^2+\|\mu-\rho+\rho\|^2 -\|\{\mu-\rho\}+\rho\|^2\\
&=\|\lambda-s\lambda\|^2-2\langle\{\mu-\rho\}-(\mu-\rho), \rho \rangle.
\end{align*}
Therefore, to understand $\Delta_1(\lambda)$, we should pay attention to $\|\lambda-s\lambda\|^2$ and the way that $\mu-\rho$ is conjugated to the dominant Weyl chamber. For convenience, we set
\begin{equation}\label{f-lambda}
f(\lambda):=\|\lambda-s\lambda\|^2,
\end{equation}
and
\begin{equation}\label{g-lambda}
g(\lambda):=2\langle\{\mu-\rho\}-(\mu-\rho), \rho \rangle.
\end{equation}
Then as deduced above
\begin{equation}\label{Delta1-f-g}
\Delta_1(\lambda)=f(\lambda)-g(\lambda).
\end{equation}

\begin{lemma}\label{lemma-f-lambda}
Fix an involution $s\in W$ such that $I(s)$ is empty.
The function $f(\lambda)$ is a homogeneous quadratic polynomial in terms of $\lambda_i$, $1\leq i\leq l$. Moreover, each term $\lambda_i^2$ has a positive coefficient, while each term $\lambda_{i} \lambda_{j}$, where $i\neq j$, has a nonnegative coefficient.
\end{lemma}
\begin{proof}
Let $s=s_{\gamma_{1}}s_{\gamma_{2}}\cdots s_{\gamma_{n}}
$ be any reduced decomposition of $s$ into simple root reflections.
Again, by Lemma 5.5 of \cite{DH11},
\begin{equation}\label{lambda-s-lambda}
\lambda-
s \lambda=\sum_{k=1}^{n} \langle
\lambda, \check{\gamma_{k}}\rangle\,
s_{\gamma_1}s_{\gamma_2}\cdots s_{\gamma_{k-1}}(\gamma_{k}),
\end{equation}
where $s_{\gamma_1}s_{\gamma_2}\cdots s_{\gamma_{k-1}}(\gamma_{k})$
is a positive root for each $k$. Thus if we set
$$
\lambda-s\lambda=\sum_{i=1}^{l} \mu_i \alpha_i,
$$
where $\mu_i=\sum_{j=1}^{l}a_{ij} \lambda_j$, then each coefficient $a_{ij}$ is a nonnegative integer. Since the set $I(s)$ is empty by assumption, according to Lemma \ref{lemma-s}, the simple root $\alpha_i$ occurs at least once in the multi-set $\{\gamma_1, \gamma_2, \dots, \gamma_n\}$. Let $k$ be the smallest index such that $\gamma_k=\alpha_i$. Then the $k$-th term of the RHS of \eqref{lambda-s-lambda} is simply
$$
s_{\gamma_1}s_{\gamma_2}\cdots s_{\gamma_{k-1}}(\alpha_i).
$$
This term will contribute $\alpha_i$ to $\lambda-s\lambda$. Thus we actually have $a_{ii}>0$.

To sum up, we have
\begin{align*}
\|\lambda-s\lambda\|^2 &=\| \lambda\|^2+\| s\lambda\|^2-2\langle \lambda, s\lambda\rangle\\
&=2\langle \lambda, \lambda-s\lambda\rangle\\
&=2\langle \sum_i\lambda_i \varpi_i, \sum_i \mu_i \alpha_i\rangle\\
&=\sum_{i} \|\alpha_i\|^2  \lambda_i \mu_i,
\end{align*}
where we use $\langle \varpi_i, \check{\alpha_j}\rangle=\delta_{ij}$ and $\alpha_i=\frac{\|\alpha_i\|^2}{2}\check{\alpha_i}$ at the penultimate step. Now the desired result follows  since $\mu_i=\sum_{j}a_{ij} \lambda_j$.
\end{proof}

\begin{lemma}\label{lemma-g-lambda}
Fix an involution $s\in W$.
The function $g(\lambda)$ is bounded over $\Lambda(s)$.
\end{lemma}
\begin{proof}
Since $\{\mu-\rho\}-(\mu-\rho)$ is an $\bbN$-combination of simple roots, it follows that $g(\lambda)$ is  bounded below by $0$. On the other hand, let $w\in W$ be such that $\{\mu-\rho\}=w(\mu-\rho)$, and let $w=s_{\gamma_1}\cdots s_{\gamma_n}$ be a reduced expression. Similar to Lemma 5.5 of \cite{DH11}, we have
\begin{align*}
\{\mu-\rho\}-(\mu-\rho)&=w(\mu-\rho)-(\mu-\rho)\\
&=\sum_{k=1}^{n} \langle
\rho-\mu, \check{\gamma_{k}}\rangle\,
s_{\gamma_1}s_{\gamma_2}\cdots s_{\gamma_{k-1}}(\gamma_{k}),
\end{align*}
where $s_{\gamma_1}s_{\gamma_2}\cdots s_{\gamma_{k-1}}(\gamma_{k})$
is a positive root for each $k$. Since $\mu$ is dominant, we have $$\langle
\rho-\mu, \check{\gamma_{k}}\rangle\leq 1, \quad 1\leq k \leq n.$$
Moreover, we have
$$
\Phi(w^{-1})=\left\{s_{\gamma_1}s_{\gamma_2}\cdots s_{\gamma_{k-1}}(\gamma_{k})\mid 1\leq k\leq n\right\}.
$$
Here $\Phi(w^{-1}):=\{\alpha\in\Delta^+\mid w^{-1}(\alpha)\in\Delta^-\}$. Therefore, to sum up,
$$
0\leq g(\lambda)\leq 2 \sum_{\alpha\in\Phi(w^{-1})}\langle\rho, \alpha \rangle\leq \|2\rho\|^2.
$$
\end{proof}

\subsection{Proof of Theorem A}
\begin{prop}\label{prop-I(s)-empty}
Fix an involution $s\in W$ such that $I(s)$ is empty.
The there are at most finitely many unitary representations in the $s$-family.
\end{prop}
\begin{proof}
By Dirac inequality, the representation $J(\lambda, -s\lambda)$ is non-unitary whenever
$\Delta_1(\lambda)>0$.
Now by \eqref{Delta1-f-g}, Lemmas \ref{lemma-f-lambda} and \ref{lemma-g-lambda}, $\Delta_1(\lambda)\leq 0$ holds for at most finitely many points in $\Lambda(s)$. The result follows.
\end{proof}

\begin{prop}\label{prop-I(s)-non-empty}
Fix an involution $s\in W$ such that $I(s)$ is non-empty.
Then  this $s$-family either contains no representations in $\widehat{G}^{\mathrm{d}}$, or it contains finitely many $I(s)$-strings of representations in $\widehat{G}^{\mathrm{d}}$. In the latter case, the strings  are all cohomologically induced from modules of $\widehat{L}_s^{\mathrm{d}}$ sitting in the $s$-family of $L_s$, and they are all in the good range. Here $L_s$ is the  Levi factor of the $\theta$-stable parabolic subgroup $P_s$ of $G$ corresponding to the simple roots $\{\alpha_i\mid i\notin I(s)\}$.
\end{prop}
\begin{proof}
Let $\frp_s=\frl_s+\fru_s$ be the Levi decomposition of the complexified Lie algebra of $P_s$. Put $S:=\text{dim}\,(\fru_s\cap\frk)$.
Suppose that this $s$-family contains members of $\widehat{G}^{\mathrm{d}}$. Take such a representation $J(\lambda_0, -s\lambda_0)$ arbitrarily. Since $2\lambda_0$ is dominant integral and regular, we have
\begin{equation}
\langle  (\lambda_0, -\lambda_0), \alpha\rangle>0, \quad \forall \alpha \in \Delta(\fru_s).
\end{equation}
Thus the good range condition is met and by Theorem \ref{thm-Vogan},
$$
J(\lambda_0, -s\lambda_0)\cong \caL_S(Z_{\lambda_0}),
$$
where $Z_{\lambda_0}$ is the irreducible unitary representation of $L_s$ with Zhelobenko parameters
$$
(\lambda_0-\frac{\rho(\fru_s)}{2}, -s(\lambda_0-\frac{\rho(\fru_s)}{2}))
=(\lambda_0-\frac{\rho(\fru_s)}{2}, -s\lambda_0+\frac{\rho(\fru_s)}{2}).
$$
Namely, $Z_{\lambda_0}$ has $T$-parameter $\lambda_0+s\lambda_0-\rho(\fru_s)$ and $A$-parameter $\lambda_0-s\lambda_0$, respectively. Moreover, by Theorem \ref{thm-D}, the inducing module $Z_{\lambda_0}$  has nonzero Dirac cohomology.

Now take any $\lambda$ in the $I(s)$-string where $\lambda_0$ sits in, i.e.,
$$\lambda-\lambda_0=\sum_{i\in I(s)}a_i \varpi_i,$$
where $2a_i\in\bbZ$. Recalling the definition of $I(s)$ in \eqref{I-s} leads to \begin{equation}\label{lambda-s-lambda0}
s(\lambda-\lambda_0)=\lambda-\lambda_0.
\end{equation}
This time the good range condition is also met since
\begin{equation}
\langle  (\lambda, -\lambda), \alpha\rangle>0, \quad \forall \alpha \in \Delta(\fru_s).
\end{equation}
Thus by Theorem \ref{thm-Vogan},
$$
J(\lambda, -s\lambda)\cong \caL_S(Z_{\lambda}),
$$
where $Z_{\lambda}$ is the irreducible representation of $L_s$ with $T$-parameter $\lambda+s\lambda-\rho(\fru_s)$ and $A$-parameter $\lambda-s\lambda$.

By \eqref{lambda-s-lambda0}, we have
$$
(\lambda+s\lambda-\rho(\fru_s))-(\lambda_0+s\lambda_0-\rho(\fru_s))
=2(\lambda-\lambda_0)=\sum_{i\in I(s)} 2 a_i \varpi_i.
$$
and
$$
\lambda-s\lambda=\lambda_0-s\lambda_0.
$$
Hence $Z_{\lambda}$ and $Z_{\lambda_0}$ differ from each other by an integral central unitary character. Therefore, $Z_{\lambda}$ is also unitary and it has nonzero Dirac cohomology as well. By Theorems \ref{thm-Vogan} and \ref{thm-D}, $J(\lambda, -s\lambda)$ is a member of $\widehat{G}^{\mathrm{d}}$. We conclude that all elements of the $I(s)$-string containing $\lambda_0$ belongs to $\widehat{G}^{\mathrm{d}}$, and they are all in the good range.

To prove that there exist at most  finitely many such $I(s)$-strings, it suffices to work on the $s$-family of $L_s$. Then on the $L_s$ level, the set $I(s)$ becomes empty, and Proposition \ref{prop-I(s)-empty} applies.
\end{proof}

 By the proof of Proposition \ref{prop-I(s)-non-empty}, the string part of $\widehat{G}^{\mathrm{d}}$ comes from the scattered parts of $\widehat{L}_\mathrm{ss}^{\mathrm{d}}$.  Here $L\supseteq HA$ runs over the proper $\theta$-stable Levi subgroups of $G$, and $L_{\mathrm{ss}}$ denotes the semisimple part of $L$. Theorem A now follows from Propositions \ref{prop-I(s)-empty} and \ref{prop-I(s)-non-empty}.

\section{Certain families of representations}\label{sec-several-families}

This section aims to study the Dirac cohomology of tempered representations, minimal representations and model representations.
\subsection{Tempered representations}

\begin{prop}\label{prop-tempered}
Let $G$ be a connected complex Lie group. Then the tempered representations with nonzero Dirac cohomology are precisely $J(\lambda, -\lambda)$, where $2\lambda$ is dominant integral and regular.
\end{prop}
\begin{proof}
By \eqref{BP} and Theorem \ref{thm-Zh}(c), the claim boils down to considering $J(\lambda, -\lambda)$, where $2\lambda$ is dominant integral and regular. Take any $K$-type $\delta$ in $J(\lambda, -\lambda)$. By Frobenius reciprocity, we have that $\delta-2\lambda $ is a positive integer combination of certain positive roots. Thus by \eqref{spin-lambda}, we have
$$
\|\delta\|_{\mathrm{spin}}\geq  \|\delta\|\geq \|2\lambda\|=\|2\lambda\|_{\mathrm{spin}}.
$$
This shows that the spin norm of $J(\lambda, -\lambda)$ is $\|2\lambda\|$, and it is achieved only on the lowest $K$-type $2\lambda$. Since $2\lambda$ is also the infinitesimal character of $J(\lambda, -\lambda)$,  it follows from Proposition \ref{prop-D-spin-lowest} that this representation  has nonzero Dirac cohomology.
\end{proof}

Tempered representations with nonvanishing Dirac cohomology have been classified in \cite{DH2} for real reductive Lie groups in Harish-Chandra class. Theorem 1.2 of \cite{DD} says that by taking the unique lowest $K$-type, these representations are in bijection with those $K$-types whose spin norm equal to their lambda norm. By the above proposition, we can  write down this bijection explicitly for complex Lie groups:
\begin{equation}
J(\lambda, -\lambda)\longleftrightarrow 2\lambda.
\end{equation}

\subsection{Minimal representations}
The minimal representations $\pi_{\mathrm{min}}$ are those attached to the minimal nilpotent coadjiont orbits of $\frg$. By \cite{Vog81}, they are ladder representations. Namely, their $K$-types are multiplicity-free and form exactly the pencil $P(0)$. It is well-known that these representations are all unitary. The following table, which is based on results of Joseph \cite{J}, gives the parameters for them.

\begin{center}
\begin{tabular}{l|r}
Type &   $\lambda_L=\lambda_R$ \\ \hline
 $A_{2n+1}$ & $\rho-\varpi_{n+1}$ \\
 $A_{2n}$ & $\rho-\frac{1}{2}(\varpi_{n}+\varpi_{n+1})$ \\
 $B_n$ & $\rho-\frac{1}{2}(\varpi_{n-2}+\varpi_{n-1})$ \\
 $C_n$ & $\rho-\frac{1}{2}\varpi_n$ \\
 $D_n$ &$\rho-\varpi_{n-2}$ \\
 $E_6, E_7, E_8$ & $\rho-\varpi_4$ \\
 $F_4$ & $\rho-\frac{1}{2}(\varpi_3+\varpi_4)$ \\
 $G_2$ & $\rho-\frac{2}{3}\varpi_2$ \\
\end{tabular}
\end{center}

\begin{prop}\label{prop-minimal}
Let $G$ be a connected complex simple Lie group. Then the minimal representation of $G$ has nonzero Dirac cohomology if and only if $G$ is $A_{2n}$, $B_n$, $C_{2n}$ or $F_4$.
\end{prop}
\begin{proof}
According to \eqref{BP} and the above table, it suffices to consider $A_{2n}$, $B_n$, $C_n$ and $F_4$. Denote by $\lambda$ the parameters $\lambda_L=\lambda_R$ in the above table.

One calculates that $P_0=\|n\beta\|_{\mathrm{spin}}=\|2\lambda\|$ for $A_{2n}$, that $P_0=\|(n-1)\beta\|_{\mathrm{spin}}=\|2\lambda\|$ for $B_{n}$,
that $P_0=\|n\beta\|_{\mathrm{spin}}=\|2\lambda\|$ for $C_{2n}$, and that $P_0=\|4\beta\|_{\mathrm{spin}}=\|2\lambda\|$ for $F_{4}$. Thus in all these cases $H_D(\pi_{\mathrm{min}})\neq 0$ by Proposition \ref{prop-D-spin-lowest}.

For $C_{2n-1}$, we have $P_0=\|n\beta\|_{\mathrm{spin}}=\|(n-1)\beta\|_{\mathrm{spin}}>\|2\lambda\|$. Thus $H_D(\pi_{\mathrm{min}})$ vanishes.
\end{proof}

\subsection{Model representations}\label{subsec-model}
The model representations are $\pi_{\mathrm{mod}}=J(\frac{\rho}{2}, \frac{\rho}{2})$.
 By Theorem 2.1 of McGovern \cite{Mc}, $\pi_{\mathrm{mod}}|_K$ is multiplicity-free and it consists exactly of those self-dual $K$-types $\delta$ such that $\delta$ lies in the root lattice.
The following result is elementary.

\begin{lemma}\label{lemma-rho-rt}
Let $G$ be a connected complex simple Lie group. Then $\rho$ lies in the root lattice of $G$ if and only if $G$ is $A_{2n}$, $C_{4n-1}$, $C_{4n}$, $D_{4n}$, $D_{4n+1}$, $G_2$, $F_4$, $E_6$ or $E_8$.
\end{lemma}

We remark that the model representation may or may not be unitary. For instance, it is unitary for $G_2$ and $E_6$, while not unitary for $C_3$, $C_4$ and $F_4$. Since $2\lambda=\rho$ for $\pi_{\mathrm{mod}}$, whenever it is unitary, we have that $H_D(\pi_{\mathrm{mod}})$ is nonzero if and only if $\rho$ occurs as a $K$-type in $\pi_{\mathrm{mod}}$. Since $\rho$ is self-dual, the latter happens if and only if $G$ is in the list of Lemma \ref{lemma-rho-rt}.

\section{The spherical unitary dual}\label{sec-spherical}

This section aims to study $\widehat{G}^{\mathrm{sd}}$, the set of \emph{non-trivial} representations with nonzero Dirac cohomology in the spherical unitary dual of $G$.
We emphasize that since the trivial representation has been excluded, the set $\widehat{G}^{\mathrm{sd}}$ could be empty. As we shall see, spherical representations live in the $w_0$-family.

\subsection{Reduction to finitely many candidates}
As mentioned earlier, for the study of Dirac cohomology, it suffices to consider
$$
J(\lambda, -s \lambda)
$$
where $s\in W$ is an involution, and $2\lambda$ is dominant integral and regular.
This representation has lowest $K$-type $\{\lambda+s\lambda\}$. Thus for it to be spherical, we must have
$s\lambda=-\lambda$. Since $\lambda$ is regular, this forces $s=w_0$, the longest element of $W$. Indeed, taking any positive root $\alpha$, we have
$$
\langle \lambda, s(\alpha)\rangle=\langle s\lambda, \alpha\rangle=\langle -\lambda, \alpha\rangle<0.
$$
Therefore, $s(\alpha)$ is a negative root. This shows that $s=w_0$.
The following result is well-known, see \cite{Hum}.

\begin{lemma}\label{lemma-w0}
Let $G$ be a connected complex simple Lie group. Then $w_0=-1$ if and only if $G$ is $A_{1}$, $B_n$, $C_n$, $D_{2n}$, $G_2$, $F_4$, $E_7$ or $E_8$.
\end{lemma}

 Except for the trivial representation, any irreducible spherical unitary representation of $G$ must be infinite dimensional.  Thus it must contain $P(0)$, the pencil starting from the trivial $K$-type.

Thus in view of Proposition \ref{prop-D-spin-lowest}, we should have
 $$\| 2\lambda\|\leq P_0.$$  Moreover, there should exist a $K$-type $\delta$ in
 $J(\lambda, -s \lambda)$ such that
 $$
 \{\delta-\rho\}+\rho=2\lambda.
 $$
One sees easily that the LHS above equals $\delta+\sum_{i} n_i \alpha_i$, where $n_i$ are some non-negative integers. By Frobenius reciprocity and the highest weight theorem, $\delta$ lies in the root lattice. We conclude that $2\lambda$ must lie in the root lattice.

To sum up, to find all the non-trivial representations with nonzero Dirac cohomology in the spherical unitary dual, it suffices to consider
\begin{equation}\label{BP-spherical}
J(\lambda, \lambda),
\end{equation}
where  $2\lambda$ is dominant integral and regular, and such that
\begin{itemize}
\item[a)] $\|2\lambda\|\leq P_0$;
\item[b)] $2\lambda$ lies in the root lattice;
\item[c)] $w_0\lambda=-\lambda$.
\end{itemize}
These requirements reduce the candidates to finitely many ones. The following table summarizes the information for some examples. Here the second row denotes the number of representations described in \eqref{BP-spherical}.

\begin{center}
\begin{tabular}{c|c|c|c|c|c|c|c|c}
$A_6$ &   $B_6$ & $C_6$ & $D_6$ & $E_6$ & $E_7$ & $E_8$ & $F_4$ &  $G_2$ \\ \hline
  $9$ & $28$ & $167$ & $18$ & $11$ & $116$ & $1080$ & $8$ & $2$\\
\end{tabular}
\end{center}

 The reduction above allows us to understand  $\widehat{G}^{\mathrm{sd}}$ in examples via using \texttt{atlas}. We will simply refer to $J(\lambda, \lambda)$ by $2\lambda$, which is expressed in terms of fundamental weights. That is, $2\lambda=[n_1, \cdots, n_l]$ means $2\lambda=\sum_{i=1}^{l} n_i \varpi_i$.

\subsection{Classical groups}
Let us present our calculations for some classical groups.

\begin{lemma}\label{lemma-spherical-classical} We have the following.
\begin{itemize}
\item[a)] $\widehat{A}_n^{\mathrm{sd}}$ is empty for $n=1, 3, 5$, while
$\widehat{A}_2^{\mathrm{sd}}=\{[1, 1]\}$,
$$\widehat{A}_4^{\mathrm{sd}}=\{[1, 1, 1, 1], [2, 1, 1, 2]\},
$$
and
$$
\widehat{A}_6^{\mathrm{sd}}=\{[1, 1, 1, 1, 1, 1], [2, 1, 1, 1, 1, 2], [2, 2, 1, 1, 2, 2]\}.
$$

\item[b)] $\widehat{B}_3^{\mathrm{sd}}=\{\pi_{\mathrm{min}}=[1,1,2 ]\}$, $\widehat{B}_4^{\mathrm{sd}}=\{\pi_{\mathrm{min}}=[2, 1, 1, 2], [1, 1, 1, 2]\}$,
     $$
    \widehat{B}_5^{\mathrm{sd}}=\{\pi_{\mathrm{min}}=[2, 2, 1, 1, 2], [1, 1, 1, 1, 2]\}.
    $$
    and
    $$
    \widehat{B}_6^{\mathrm{sd}}=\{\pi_{\mathrm{min}}=[2, 2, 2, 1, 1, 2], [2, 1, 1, 1, 1, 2], [1, 1, 1, 1, 1, 2]\}.
    $$

\item[c)] $\widehat{C}_n^{\mathrm{sd}}$ is empty for $n=3, 5$, while
$\widehat{C}_n^{\mathrm{sd}}=\{\pi_{\mathrm{min}}\}$ for $n=2, 4, 6$.

\item[d)] $\widehat{D}_n^{\mathrm{sd}}=\{\pi_{\mathrm{mod}}\}$ for $n=4, 5$, while $\widehat{D}_6^{\mathrm{sd}}=\{[2, 1, 1, 1, 1, 1]\}$.
\end{itemize}
All of them are $K$-multiplicity free.
\end{lemma}
We remark that all the representations in the above lemma are unipotent ones. Indeed, Section 5 of \cite{BP} offers excellent interpretations for them.

\subsection{Exceptional groups}

\begin{example}\label{exam-G2} Let us consider $G_2$. In this case, there are two representations
meeting the requirements of \eqref{BP-spherical}:
$$[1, 1], [2, 1].$$
Then \texttt{atlas} calculates that the first representation is unitary, while the second one is not. For instance, we put the first representation into \texttt{atlas} via the following commands:
\begin{verbatim}
set G=complex(simply_connected(G2))
set x=x(trivial(G))
set p=param(x, [0,0,0,0],[1,1,0,0])
\end{verbatim}
To test its unitarity, we use the command
\begin{verbatim}
is_unitary(p)
\end{verbatim}
The output is
\begin{verbatim}
Value: true
\end{verbatim}
The $K$-type $[1, 1]$ has \texttt{atlas} height $16$, and the following command looks at its $K$-types up to this height:
\begin{verbatim}
branch_irr(p, 16)
\end{verbatim}
The output is
\begin{verbatim}
Value:
1*parameter(x=0,lambda=[0,0,0,0]/1,nu=[0,0,0,0]/1) [0]
1*parameter(x=0,lambda=[1,0,0,0]/1,nu=[0,0,0,0]/1) [6]
1*parameter(x=0,lambda=[0,1,0,0]/1,nu=[0,0,0,0]/1) [10]
1*parameter(x=0,lambda=[1,0,1,0]/1,nu=[0,0,0,0]/1) [12]
1*parameter(x=0,lambda=[1,1,0,0]/1,nu=[0,0,0,0]/1) [16]
\end{verbatim}
The last line above tells us that the $K$-type $[1, 1]$ occurs with multiplicity one in the model representation. (Of course, this was already known by McGovern \cite{Mc} in 1994.) Since
$$
\|\rho\|_{\mathrm{spin}}=\|\rho\|,
$$
we conclude from Proposition \ref{prop-D-spin-lowest} that $\widehat{G}_2^{\mathrm{sd}}=\{\pi_{\mathrm{mod}}\}$.  \qed
\end{example}

\begin{example}\label{exam-F4} Let us consider $F_4$. In this case, there are eight representations
meeting the requirements of \eqref{BP-spherical}:
\begin{eqnarray*}
&[2, 2, 1, 1], \quad [1, 2, 1, 1], \quad [2, 1, 1, 1], \quad [2, 1, 1, 2],\\
&[1, 1, 2, 1], \quad [1, 1, 1, 2], \quad [1, 1, 1, 1], \quad [3, 1, 1, 1].
\end{eqnarray*}
By Proposition \ref{prop-minimal}, $[2, 2, 1, 1]$ is the minimal representation and has nonzero Dirac cohomology. \texttt{atlas} calculates that the seven remaining representations are not unitary. We conclude that $\widehat{F}_4^{\mathrm{sd}}=\{\pi_{\mathrm{min}}\}$.
\qed
\end{example}

\begin{example}\label{exam-E6} Let us consider $E_6$. In this case, there are eleven representations
meeting the requirements of \eqref{BP-spherical}. Among them, $[1, 1, 1, 1, 1, 1]$ stands for the model representation, and it has nonzero Dirac cohomology by the discussion in \S \ref{subsec-model}. \texttt{atlas} calculates the $K$-types pattern of the other ten representations, and we deduce from Parthasarathy's Dirac inequality that they are not unitary. Details are given in the following table.
\begin{center}
\begin{tabular}{c|c}
$2\lambda$ &   $K$-type $\delta$  \\
\hline
$[1, 1, 2, 1, 2, 1]$ & $[1, 0, 0, 1, 0, 1]$ \\
$[1, 4, 1, 1, 1, 1]$ & $[1, 1, 0, 0, 0, 1]$ \\
$[3, 1, 1, 1, 1, 3]$ & $[1, 1, 0, 0, 0, 1]$ \\
$[1, 2, 1, 2, 1, 1]$ & $[1, 1, 0, 0, 0, 1]$ \\
$[2, 1, 1, 2, 1, 2]$ & $[1, 0, 1, 1, 0, 0]$ \\
$[1, 2, 1, 1, 1, 1]$ & $[1, 0, 1, 1, 1, 1]$ \\
$[2, 2, 1, 1, 1, 2]$ & $[0, 0,  1, 0, 1, 0]$ \\
$[2, 1, 1, 1, 1, 2]$ & $[1, 0, 1, 0, 1, 1]$ \\
$[1, 1, 1, 2, 1, 1]$ & $[1, 1, 1, 1, 0, 0]$ \\
$[1, 3, 1, 1, 1, 1]$ & $[1, 0, 1, 1, 0, 0]$
\end{tabular}
\end{center}
The second column of the table above specifies a $K$-type $\delta$ in $J(\lambda, \lambda)$
such that $$P_{\delta}<\|2\lambda\|.$$
We conclude that $\widehat{E}_6^{\mathrm{sd}}=\{\pi_{\mathrm{mod}}\}$.
\qed
\end{example}

\subsection{A conjecture}
The previous calculation leads us to make the following.
\begin{conj}\label{conj-spherical-classical}
The set $\widehat{G}^{\mathrm{sd}}$ can be described as follows.
\begin{itemize}
\item[a)] $\widehat{A}_{2n-1}^{\mathrm{sd}}$ is empty, while
$\widehat{A}_{2n}^{\mathrm{sd}}$ consists of the following $n$ representations:
$$
[\underbrace{2,\dots, 2}_{p}, \underbrace{1, \dots, 1}_{2n-2p}, \underbrace{2,\dots, 2}_{p}], \quad 0\leq p\leq n-1.
$$

\item[b)] $\widehat{B}_n^{\mathrm{sd}}$ consists of the following $[\frac{n}{2}]$ representations:
    $$
[\underbrace{2,\dots, 2}_{p(a, b)}, \underbrace{1, \dots, 1,}_{n-p(a, b)-1} 2],
$$
where $a+b=n$, $b\geq a\geq 1$, and $p(a, b):=\max\{b-a-1, 0\}$.

\item[c)] $\widehat{C}_{2n-1}^{\mathrm{sd}}$ is empty, while
$\widehat{C}_{2n}^{\mathrm{sd}}=\{\pi_{\mathrm{min}}\}$, where $\pi_{\mathrm{min}}$ stands for the minimal representation.

\item[d)] $\widehat{D}_n^{\mathrm{sd}}$ consists of the following  $[\frac{n}{4}]$ representations:
$$
[\underbrace{2,\dots, 2}_{p(a, b)}, \underbrace{1, \dots, 1}_{n-p(a, b)}],
$$
where $a+b=n$, $b\geq a\geq 2$, $a$ is even, and $p(a, b):=\max\{b-a-1, 0\}$.
\end{itemize}
In particular, any representation in $\widehat{G}^{\mathrm{sd}}$ is $K$-multiplicity free.
\end{conj}

Thanks to the work carried out in Section 5 of \cite{BP}, all the representations above are unipotent ones with nonzero Dirac cohomology, and all of them are $K$-multiplicity free.
Thus the hard part of the conjecture is to show that these representations \emph{exhaust} the set $\widehat{G}^{\mathrm{sd}}$. That is, $\widehat{G}^{\mathrm{sd}}$ should contain no other representation. Although the unitary dual of classical complex Lie groups has been described by Vogan \cite{Vog86} and Barbasch \cite{B}, this still seems to be rather non-trivial (for the authors). For exceptional groups, we would like to guess that $\widehat{E}_7^{\mathrm{sd}}$ is empty, while $\widehat{E}_8^{\mathrm{sd}}=\{\pi_{\mathrm{mod}}\}$.

\section{A computing method}\label{sec-comp}

This section aims to introduce a method that allows us to compute all the members of $\widehat{G}^{\mathrm{d}}$ in any $s$-family such that $I(s)$ is empty. Thus by Proposition \ref{prop-I(s)-non-empty}, eventually we can compute the entire $\widehat{G}^{\mathrm{d}}$.
Our basic idea is to use Parthasarathy's Dirac inequality and Vogan pencil.
More precisely, we proceed as follows:
\begin{itemize}
\item[$\bullet$] calculate the lowest $K$-type $\mu:=\{\lambda+s\lambda\}$ for $J(\lambda, -s\lambda)$.
\item[$\bullet$] when $\lambda$ is large, calculate
\begin{equation}\label{Delta-1}
\Delta_1(\lambda):= \|2\lambda\|^2-\|\mu\|_{\mathrm{spin}}^2.
\end{equation}
\item[$\bullet$] when $\lambda$ is small, calculate
\begin{equation}\label{Delta-2}
\Delta_2(\lambda):= \|2\lambda\|^2-P_\mu^2.
\end{equation}
\end{itemize}
Here as in \eqref{P-mu-prime}, $P_\mu$ is the minimal spin norm of the $K$-types lying on $P(\mu)$---the Vogan pencil starting from $\mu$.  Note that $\Delta_2(\lambda)\leq 0$ sharpens $\Delta_1(\lambda)\leq 0$ whenever $\mu$ is u-small, see \eqref{P-mu}. We call them \emph{discriminants} for $\lambda$. Whenever either discriminant is positive, the representation  $J(\lambda, -s\lambda)$ is non-unitary.
To have more flexibility, we shall just leave the  precise description of ``large" and ``small" blank. However, looking at the boundary of the u-small convex hull is always helpful.

Since $I(s)$ is empty, Proposition \ref{prop-I(s)-empty} guarantees that we are left with at most finitely many candidate representations. Then by \texttt{atlas} \cite{At}, one can eventually handle them completely.

Let us illustrate this method carefully for $G_2$, whose unitary dual was determined by Duflo \cite{Du} in 1979. We denote by $\alpha_1$ the short simple root, while $\alpha_2$ is long. There are eight involutions in the Weyl group of $G_2$. Let $\lambda=[a, b]$, where $a, b\in \frac{1}{2}\mathbb{P}$. We have the following table.

\begin{center}
\begin{tabular}{r|c|r}
Involution $s$ &   $\lambda+s\lambda$ & $I(s)$ \\ \hline
 $e$ & $[2a, 2b]$ & $\{1, 2\}$ \\
 $s_1$ & $[0, a+2b]$ & $\{2\}$ \\
 $s_2$ & $[2a+3b, 0]$ & $\{1\}$\\
 $s_1 s_2 s_1$ & $[-a-3b, a+3b]$ & $\emptyset$ \\
 $s_2 s_1 s_2$ &$[3a+3b, -a-b]$ & $\emptyset$ \\
 $s_1 s_2 s_1 s_2 s_1$ & $[-3b, 2b]$ & $\emptyset$\\
 $s_2 s_1 s_2 s_1 s_2$ & $[2a, -a]$ & $\emptyset$\\
 $w_0$ & $[0, 0]$ & $\emptyset$\\
\end{tabular}
\end{center}

The three $s$-families where $I(s)$ are non-empty can be handled easily: it boils down to work on the $s$-family of the  corresponding Levi factors. The $w_0$-family has been considered in Example \ref{exam-G2}. Let us focus on the four remaining $s$-families one by one.

For $s=s_2 s_1 s_2 s_1 s_2$, we have $\mu=[a, 0]$. Therefore, $a$ must be a positive integer. When $a\geq 4$, we have $\{\mu-\rho\}=[a-4, 1]$. Then
$$
\Delta_1(\lambda)=6 a^2 + 24 a b + 24 b^2-6>0.
$$
Thus these representations are not unitary. One can also calculate that
\begin{equation*}
\Delta_2(\lambda)=
\begin{cases}
24 b^2+72 b + 30 & \mbox{ if } a=3; \\
24 b^2+ 48 b + 6 & \mbox{ if } a=2;\\
24 b^2 + 24 b -6 & \mbox{ if } a=1,
\end{cases}
\end{equation*}
which is always positive. Thus the corresponding representations are not unitary either.

For $s=s_1 s_2 s_1 s_2 s_1$, we have $\mu=[0, b]$. Therefore, $b$ must be a positive integer. When $b\geq 2$, we have $\{\mu-\rho\}=[1, b-2]$. Then
$$
\Delta_1(\lambda)=8 a^2 + 24 a b + 18 b^2-2>0.
$$
Thus these representations are not unitary. When $b=1$, one can also calculate that
$$
\Delta_2(\lambda)=8 a^2 + 24 a -2>0.
$$
Thus these representations are not unitary either.

For $s=s_1 s_2 s_1$, we have $\mu=[a+3b, 0]$. Therefore, $a+3b$ must be an integer. When $a+3b\geq 4$, we have $\{\mu-\rho\}=[a+3b-4, 1]$. Then
$$
\Delta_1(\lambda)=6 a^2 + 12 a b + 6 b^2-6 >0.
$$
Thus these representations are not unitary. When $a+3b<4$, then we must have $a=\frac{3}{2}$ and $b=\frac{1}{2}$, or $a=b=\frac{1}{2}$. Thus it remains to consider the representations
$$
J([\frac{3}{2}, \frac{1}{2}], [\frac{9}{2}, -\frac{5}{2}]), \quad J([\frac{1}{2}, \frac{1}{2}], [\frac{5}{2}, -\frac{3}{2}]).
$$
By \texttt{atlas}, the first one is not unitary, while the second one is unitary. Moreover, the latter representation has the unique spin lowest $K$-type $\rho$ such that $\|\rho\|_{\rm spin}=\|\rho\|$. Thus it belongs to $\widehat{G}_2^{\rm d}$.

For $s=s_2 s_1 s_2$, we have $\mu=[0, a+b]$.  When $a+b\geq 2$, we have $\{\mu-\rho\}=[1, a+b-2]$. Then
$$
\Delta_1(\lambda)=2 a^2 + 12 a b + 18 b^2-2 >0.
$$
Thus these representations are not unitary.
When $a+b=1$, we must have $a=b=\frac{1}{2}$, and the representation is
$$
J([\frac{1}{2}, \frac{1}{2}], [-\frac{5}{2}, \frac{3}{2}]).
$$
It is non-unitary by \texttt{atlas}.

To sum up, the set $\widehat{G}_2^{\mathrm{d}}$ is pinned down as follows, where $a, b\in\frac{1}{2}\mathbb{P}$. Note that it consists of three scattered members and three strings. Note also that the last row in Table  \ref{table-G2} is the trivial representation, while the penultimate row is the model representation.

\begin{table}[H]
\caption{The set $\widehat{G}_2^{\mathrm{d}}$}
\begin{tabular}{l|c|c|r}
$s$ &   $\lambda$   & spin LKT & mult \\
\hline
$e$ & $[a, b]$ & LKT  & $1$ \\
$s_1$ & $[1, b]$ & LKT  &  $1$ \\
$s_2$ & $[a, 1]$ &  LKT  & $1$\\
$s_1s_2s_1$ & $\frac{\rho}{2}$ &  $\rho$ & $1$\\
$w_0$ & $\frac{\rho}{2}$ &  $\rho$ & $1$\\
$w_0$ & $\rho$ &  $[0, 0]$ & $1$
\end{tabular}
\label{table-G2}
\end{table}

\section{The string part of $\widehat{F}_4^{\mathrm{d}}$}

From now on, we set $G$ to be complex $F_4$, whose Dynkin diagram is in Fig.~1, where $\alpha_1$ and $\alpha_2$ are short, while $\alpha_3$ and $\alpha_4$ are long, see page 691 of Knapp \cite{Kn} for more details.
\begin{figure}[H]
\centering \scalebox{0.6}{\includegraphics{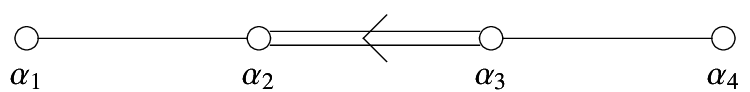}}
\caption{Dynkin diagram for $F_4$.}
\end{figure}

There are $140$ involutions in $W$, which are indexed in the Appendix. Thus we will freely refer to an involution by its index there. This section aims to figure out the string part of  $\widehat{F}_4^{\mathrm{d}}$. Namely, we shall consider the  $37$ involutions where $I(s)$ are non-empty.
Guided by Proposition \ref{prop-I(s)-non-empty}, finding members of $\widehat{F}_4^{\mathrm{d}}$ in such an $s$-family boils down to finding the members of $\widehat{L}_s^{\mathrm{d}}$ in the $s$-family of $L_s$. Since the Levi factor $L_s$ is always classical, getting the string part of $\widehat{F}_4^{\mathrm{d}}$ is relatively easier than getting its scattered part. Let us see an example.

\begin{example}
Consider the involution $s$ with index $15$, i.e., $s=s_2s_1s_3s_2$. Then $I(s)=\{4\}$.
Let $P_s$ be the $\theta$-stable parabolic subgroup of complex $F_4$ corresponding to the simple roots $\alpha_1$, $\alpha_2$ and $\alpha_3$. Let $L_s$ be its Levi subgroup. The semisimple part of $L_s$ is complex $C_3$, and one calculates that in the $s$-family of $C_3$, there is only one member of $\widehat{C}_3^{\mathrm{d}}$.
Namely, it is $J(\lambda_s, -s\lambda_s)$ for $C_3$, where $\lambda_s=[\frac{1}{2}, \frac{1}{2}, \frac{1}{2}]$. The $T$-parameter  and $A$-parameter of this representation are
$$[2, -2, 2], \quad [-1, 3, -1], $$
respectively. Moreover, $2\lambda_s$ is its unique spin lowest $K$-type occurring with multiplicity one.

By Proposition \ref{prop-I(s)-non-empty}, after tensoring with suitable central unitary characters, the module above  produces the $I(s)$-string $\lambda=[\frac{1}{2}, \frac{1}{2}, \frac{1}{2}, d]$ in $\widehat{F}_4^{\mathrm{d}}$ via cohomological parabolic induction. Here $d$ runs over $\frac{1}{2}\mathbb{P}$. Note that the $T$-parameters and $A$-parameters of this string are
$$
[2, -2, 2, 2 d+1], \quad [-1, 3, -1, -1]
$$
respectively. Accordingly, $2\lambda$ is the unique spin lowest $K$-type of $J(\lambda, -s\lambda)$ and it occurs exactly once. This explains the $18$th row of
Table \ref{table-F4-string-part}.
\hfill\qed
\end{example}

Other $s$-families are handled similarly. The final result is presented in Table \ref{table-F4-string-part}.

\begin{table}
\centering
\caption{The string part of $\widehat{F}_4^{\mathrm{d}}$}
\begin{tabular}{l|c|c|r}
\#$s$ &   $\lambda$  & spin LKT & mult \\
\hline
$1$ & $[a, b, c, d]$ & LKT& $1$ \\
$2$ & $[1, b, c, d]$ & LKT & $1$\\
$3$ &   $[a, 1, c, d]$ & LKT & $1$\\
$4$ &  $[a, b, 1, d]$ & LKT & $1$ \\
$5$ &    $[a, b, c, 1]$ & LKT & $1$\\
$6$ &    $[1, b, 1, d]$ & LKT & $1$\\
$7$ &   $[1, b, c, 1]$ & LKT & $1$\\
$8$ &  $[a, 1, c, 1]$ & LKT & $1$\\
$9$ & $[\frac{1}{2}, \frac{1}{2}, c, d]$ &  $2\lambda$ &  $1$ \\
$9$  & $[1, 1, c, d]$   & LKT &  $1$ \\
$10$ & $[a, \frac{1}{2}, \frac{1}{2}, d]$ & $2\lambda$  & $1$  \\
$12$ & $[a, b, \frac{1}{2}, \frac{1}{2}]$ & $2\lambda$  &   $1$\\
$12$ & $[a, b, 1, 1]$  & LKT  & $1$ \\
$13$ & $[\frac{1}{2}, \frac{1}{2}, c, 1]$ & $[1, 1, 2c+1, 0]$ &  $1$ \\
$13$ & $[1, 1, c, 1]$ & LKT  &  $1$ \\
$14$ & $[1, b, \frac{1}{2}, \frac{1}{2}]$ & $[0, 2b+1, 1, 1]$ & $1$\\
$14$ & $[1, b, 1, 1]$ & LKT & $1$\\
$15$ & $[\frac{1}{2}, \frac{1}{2}, \frac{1}{2}, d]$ &  $2\lambda$ &  $1$\\
$16$ &   $[a, 1, 1, d]$ & LKT & $1$\\
$16$ &   $[a, 1, \frac{1}{2}, d]$ & $[2a+1, 2, 0, 2d+1]$ & $1$\\
$20$ & $[a, \frac{1}{2}, \frac{1}{2}, 1]$ & $[2a, 3, 0,1]$ & $1$ \\
$23$ & $[\frac{1}{2}, \frac{1}{2}, 1, d]$ & $[3, 1, 0, 2d+2]$  & $1$\\
$23$ &  $[\frac{1}{2}, \frac{1}{2}, \frac{1}{2}, d]$ & $2\lambda$  &  $1$ \\
$27$ & $[a, 1, \frac{1}{2}, \frac{1}{2}]$ & $[2a+2, 0, 1, 2]$ & $1$ \\
$33$ & $[a, \frac{1}{2}, \frac{1}{2}, \frac{1}{2}]$ & $2\lambda$ &  $1$ \\
$34$ & $[1, 1, \frac{1}{2}, d]$ & $[3, 0, 0, 2d+3]$  &  $1$\\
$34$ & $[1, \frac{1}{2}, \frac{1}{2}, d]$ & $[1, 2, 0, 2d+1]$ & $1$\\
$47$ &  $[1, 1, 1, d]$ & LKT & $1$\\
$50$ & $[a, 1, 1, 1]$ & LKT & $1$ \\
$50$ & $[a, 1, \frac{1}{2}, \frac{1}{2}]$ & $[2a+2, 0, 2, 0]$ & $1$ \\
\end{tabular}
\label{table-F4-string-part}
\end{table}

\section{The scattered part of $\widehat{F}_4^{\mathrm{d}}$}

In this section, we use $\lambda=[a, b, c, d]$ to denote the weight $a\varpi_1+b\varpi_2+c\varpi_3+d\varpi_4$, where $a, b, c, d\in\frac{1}{2}\mathbb{P}$. Set $\mu:=\{\lambda+s\lambda\}$, which is the lowest $K$-type of $J(\lambda, -s\lambda)$.
Let us focus on the $103$ $s$-families where $I(s)$ are empty. To use the two discriminants $\Delta_1(\lambda)$ and $\Delta_2(\lambda)$ efficiently, we shall arrange these $s$-families according to their types, and each type essentially bears a common pattern. We will provide the common pattern for each type, and give a few examples.

The final result is presented in Table \ref{table-F4-scattered-part}.
We note that according to page 13 of \cite{BP}, there are ten representations in $\widehat{F}_4^{\mathrm{d}}$ with $\lambda=\frac{\rho}{2}$: three of them appear in  Table \ref{table-F4-scattered-part}, while the other ones are merged into seven strings in Table \ref{table-F4-string-part}. Note also that the last row of Table \ref{table-F4-scattered-part} is the trivial representation, and the penultimate row there is the minimal representation.

\begin{table}
\centering
\caption{The scattered part of $\widehat{F}_4^{\mathrm{d}}$}
\begin{tabular}{l|c|c|c|r}
\#$s$ &   $\lambda$   & spin LKT & mult &  u-small\\
\hline
$25$ & $[\frac{1}{2}, \frac{1}{2}, \frac{1}{2}, 1]$ & $[1, 3, 0, 1]$  & $1$ & Yes\\
$38$ & $\frac{\rho}{2}$ & $\rho$  &  $1$ & Yes\\
$62$ & $[1, 1, \frac{1}{2}, \frac{1}{2}]$ &  $[0, 0, 1, 4]$  & $1$  & Yes\\
$63$ & $[\frac{1}{2}, \frac{1}{2}, 1, 1]$ &  $[7, 1, 0, 0]$ & $1$  & Yes\\
$63$ & $\frac{\rho}{2}$ & $\rho$ & $1$  & Yes\\
$76$ & $[1, \frac{1}{2}, \frac{1}{2}, 1]$ &  $[4, 2, 0, 0]$ & $1$  & Yes\\
$92$ & $[1, \frac{1}{2}, \frac{1}{2}, \frac{1}{2}]$ &  $[2, 2, 0, 1]$ & $1$  & Yes\\
$122$ & $\frac{\rho}{2}$ & $\rho$ & $1$  & Yes\\
$140$ &  $[1,1,\frac{1}{2}, \frac{1}{2} ]$ & $[0,0,0,4]$ & $1$ & Yes\\
$140$ &  $\rho$ & $[0,0,0,0]$ & $1$ & Yes
\end{tabular}
\label{table-F4-scattered-part}
\end{table}

\subsection{Type $(1)$}
There are eleven involutions such that $\mu=x\varpi_1$ for any $\lambda$, where $x\in\mathbb{P}$. We refer to these $s$-families as families of type $(1)$, and adopt the following common pattern to handle them.
\begin{itemize}
\item[(a)] calculate $\Delta_1(\lambda)$ for $x\geq 10$, then $\{\mu-\rho\}+\rho=[x-9, 2, 2, 2]$.
\item[(b)] calculate $\Delta_2(\lambda)$ for the nine (possible) remaining  points:
$$
x=1, 2, 3, 4, 5, 6, 7, 8, 9.
$$
\end{itemize}

\begin{center}
\begin{tabular}{l|r}
Type &  \#s \\
\hline
$(1)$ & $63$, $76$, $92$, $109$, $110$, $120$, $122$, $130$, $132$, $138$, $139$\\
\end{tabular}
\end{center}

We give a concrete example to illustrate the above pattern.
\begin{example}
Let us consider the involution $s$ with index $63$. Then
$$
\mu=[a+3b+4c+2d, 0, 0, 0].$$
Thus $x=a+3b+4c+2d\geq 5$. When $x\geq 10$, we have
$$
\Delta_1(\lambda)=\frac{1}{3} x^2 +\frac{4}{3} ax +\left(\frac{4}{3}a^2 + \frac{8}{3} c^2 + \frac{8}{3} c d + \frac{8}{3} d^2\right)-35.
$$
The term in the bracket takes the minimal value $\frac{7}{3}$ when $a=c=d=\frac{1}{2}$. It is then easy to see that $\Delta_1(\lambda)>0$ when $x\geq 10$. Now it remains to calculate $\Delta_2(\lambda)$ for $x=5,6,7,8,9$. We only present the discussion for $x=5, 8$. In the former case, we must have $\lambda=\frac{\rho}{2}$, and it is in $\widehat{F}_4^{\mathrm{d}}$. When $x=8$, we can have $c=\frac{1}{2}$ or $1$. A little more calculation gives that there are seven choices for $\lambda$ in total, and $\Delta_2(\lambda)>0$ fails exactly in the following cases:
$$
\lambda=[\frac{1}{2}, \frac{1}{2}, 1, 1], \quad [1, 1, \frac{1}{2}, 1], \quad [\frac{1}{2}, \frac{3}{2}, \frac{1}{2}, \frac{1}{2}].
$$
\texttt{atlas} says that the first representation is unitary, while the other two are not. Then a closer look at the first representation says that it has a unique spin lowest $K$-type $[7, 1, 0, 0]$, which occurs with multiplicity one. Moreover,
 $$
 \|[7, 1, 0, 0]\|_{\mathrm{spin}}=\|2\lambda\|.
 $$
 Thus $J(\lambda, -s\lambda)\in \widehat{F}_4^{\mathrm{d}}$ for $\lambda=[\frac{1}{2}, \frac{1}{2}, 1, 1]$ by Proposition \ref{prop-D-spin-lowest}.
 \hfill\qed
\end{example}

\subsection{Type $(4)$}
There are eleven non-dominant involutions such that $\mu=x\varpi_4$ for any $\lambda$, where $x\in\mathbb{P}$. We refer to these $s$-families as families of type $(4)$, and adopt the following common pattern to handle them.
\begin{itemize}
\item[(a)] calculate $\Delta_1(\lambda)$ for $x\geq 7$, then $\{\mu-\rho\}+\rho=[2, 2, 2, x-6]$.
\item[(b)] calculate $\Delta_2(\lambda)$ for the six (possible) remaining  points:
$$
x=1, 2, 3, 4, 5, 6.
$$
\end{itemize}

\begin{center}
\begin{tabular}{l|r}
Type &  \#s \\
\hline
$(4)$ & $62$,  $77$, $93$, $108$, $111$, $121$, $123$, $129$, $131$, $136$, $137$
\end{tabular}
\end{center}

\subsection{Type $(13)$}
There are thirteen involutions such that $\mu=x\varpi_1+y\varpi_3$ for any $\lambda$, where $x, y\in\mathbb{P}$. We refer to these $s$-families as families of type $(13)$, and adopt the following common pattern to handle them.
\begin{itemize}
\item[(a)] calculate $\Delta_1(\lambda)$ for the following (possible) cases:
  \begin{itemize}
  \item[$\bullet$] $x\geq 2$ and $y\geq 3$, then $\{\mu-\rho\}+\rho=[x-1,2,y-2, 2]$.
  \item[$\bullet$] $x=1$ and $y\geq 3$, then $\{\mu-\rho\}+\rho=[2,1,y-2, 2]$.
  \item[$\bullet$] $x\geq 3$ and $y=2$, then $\{\mu-\rho\}+\rho=[x-2, 2, 1, 1]$.
  \item[$\bullet$] $x\geq 6$ and $y=1$, then $\{\mu-\rho\}+\rho=[x-5, 2,1, 2]$.
  \end{itemize}
\item[(b)] calculate $\Delta_2(\lambda)$ for the seven (possible) remaining  points:
$$
(x, y)=(1, 1), (2, 1), (3, 1), (4, 1), (5, 1), (1,2), (2, 2).
$$
\end{itemize}

\begin{center}
\begin{tabular}{l|r}
Type &  \#s \\
\hline
$(13)$ & $29$, $37$, $43$, $51$,  $53$, $54$, $66$, $67$, $82$, $83$, $101$, $102$, $112$
\end{tabular}
\end{center}

\subsection{Type $(23)$}
For the involution with index $60$ we have that $\mu=x\varpi_2+y\varpi_3$ for any $\lambda$, where $x, y\in\mathbb{P}$. We refer to this $s$-family as family of type $(23)$, and adopt the following pattern to handle it.
\begin{itemize}
\item[(a)] calculate $\Delta_1(\lambda)$ for the following (possible) cases:
  \begin{itemize}
  \item[$\bullet$] $x\geq 2$ and $y\geq 2$, then $\{\mu-\rho\}+\rho=[2,x-1,y-1,2]$.
  \item[$\bullet$] $x=1$ and $y\geq 3$, then $\{\mu-\rho\}+\rho=[1, 2, y-2, 2]$.
  \item[$\bullet$] $x\geq 4$ and $y=1$, then $\{\mu-\rho\}+\rho=[2, x-3, 2, 1]$.
  \end{itemize}
\item[(b)] calculate $\Delta_2(\lambda)$ for the four (possible) remaining  points:
$$
(x, y)=(1, 1),   (2, 1), (3, 1), (1, 2).
$$
\end{itemize}

\subsection{Type $(24)$}
There are eleven involutions such that $\mu=x\varpi_2+y\varpi_4$ for any $\lambda$, where $x,y\in\mathbb{P}$. We refer to these $s$-families as families of type $(24)$, and adopt the following common pattern to handle them.
\begin{itemize}
\item[(a)] calculate $\Delta_1(\lambda)$ for the following (possible) cases:
  \begin{itemize}
  \item[$\bullet$] $x\geq 4$ and $y\geq 2$, then $\{\mu-\rho\}+\rho=[2,x-3,2,y-1]$.
  \item[$\bullet$] $x\geq 4$ and $y=1$, then $\{\mu-\rho\}+\rho=[2,x-3,1,2]$.
  \item[$\bullet$] $x=3$ and $y\geq 2$, then $\{\mu-\rho\}+\rho=[1, 2, 1,y-1]$.
  \item[$\bullet$] $x=2$ and $y\geq 3$, then$\{\mu-\rho\}+\rho=[1, 2,1, y-2]$.
  \item[$\bullet$] $x=1$ and $y\geq 5$, then $\{\mu-\rho\}+\rho=[2,1,2, y-4]$.
  \end{itemize}
\item[(b)] calculate $\Delta_2(\lambda)$ for the seven (possible) remaining  points:
$$
(x, y)=(1, 1), (1, 2), (1, 3), (1, 4), (2, 1), (2, 2), (3, 1).
$$
\end{itemize}

\begin{center}
\begin{tabular}{l|r}
Type &  \#s \\
\hline
$(24)$ & $46$, $55$, $59$, $61$, $70$, $75$, $85$, $88$, $99$, $106$, $119$
\end{tabular}
\end{center}

\subsection{Type $(14)$}
There are thirteen involutions such that $\mu=x\varpi_1+y\varpi_4$ for any $\lambda$, where $x,y\in\mathbb{P}$. We refer to these $s$-families as families of type $(14)$, and adopt the following common pattern to handle them.
\begin{itemize}
\item[(a)] calculate $\Delta_1(\lambda)$ for the following (possible) cases:
  \begin{itemize}
  \item[$\bullet$] $x\geq 5$ and $y\geq 4$, then $\{\mu-\rho\}+\rho=[x-4,2, 2, y-3]$.
  \item[$\bullet$] $x=4$ and $y\geq 4$, then $\{\mu-\rho\}+\rho=[2,1,2, y-3]$.
  \item[$\bullet$] $x=3$ and $y\geq 4$, then $\{\mu-\rho\}+\rho=[2, 2, 1,y-3]$.
  \item[$\bullet$] $x=2$ and $y\geq 5$, then $\{\mu-\rho\}+\rho=[2, 1, 2, y-4]$.
  \item[$\bullet$] $x=1$ and $y\geq 6$, then $\{\mu-\rho\}+\rho=[1,2,2, y-5]$.
  \item[$\bullet$] $x\geq 5$ and $y=3$, then $\{\mu-\rho\}+\rho=[x-4, 2, 1,2]$.
  \item[$\bullet$] $x\geq 6$ and $y=2$, then $\{\mu-\rho\}+\rho=[x-5, 2, 1,2]$.
  \item[$\bullet$] $x\geq 8$ and $y=1$, then $\{\mu-\rho\}+\rho=[x-7, 2,2,1]$.
  \end{itemize}
\item[(b)] calculate $\Delta_2(\lambda)$ for the nineteen (possible) remaining  points:
$$
(x, 1), 1\leq x\leq 7; (x, 2), 1\leq x\leq 5; (x, 3), 1\leq x\leq 4; (x, 4), 1\leq x\leq 2; (1, 5).
$$
\end{itemize}

\begin{center}
\begin{tabular}{l|r}
Type &  \#s \\
\hline
$(14)$ & $38$, $52$, $57$, $69$, $72$, $84$, $89$, $96$, $103$, $107$, $114$, $118$, $125$
\end{tabular}
\end{center}

\subsection{Type $(134)$}
For the involution with index $11$ we have that $\mu=x\varpi_1+y\varpi_3+z\varpi_4$ for any $\lambda$, where $x,y,z\in\mathbb{P}$. We refer to this $s$-family as family of type $(134)$, and adopt the following common pattern to handle it.
\begin{itemize}
\item[(a)] calculate $\Delta_1(\lambda)$ for the following (possible) cases:
  \begin{itemize}
  \item[$\bullet$] $x\geq 2$ and $y\geq 2$, then $\{\mu-\rho\}+\rho=[x-1,2,y-1,z]$.
  \item[$\bullet$] $x\geq 3$, $y=1$ and $z\geq 2$, then $\{\mu-\rho\}+\rho=[x-2,2, 1, z-1]$.
  \item[$\bullet$] $x\geq 4$, $y=1$ and $z=1$, then $\{\mu-\rho\}+\rho=[x-3, 2, 1, 1]$.
  \item[$\bullet$] $x=1$, $y\geq 2$, then $\{\mu-\rho\}+\rho=[2, 1, y-1, z]$.

  \item[$\bullet$] $x=2$, $y=1$ and $z\geq 2$, then $\{\mu-\rho\}+\rho=[2, 1, 1, z-1]$.
  \item[$\bullet$] $x=1$, $y=1$ and $z\geq 3$, then $\{\mu-\rho\}+\rho=[1, 2, 1, z-2]$.
  \end{itemize}
\item[(b)] calculate $\Delta_2(\lambda)$ for the four (possible) remaining  points:
$$
(x, y, z)=(1,1,1), (1,1,2), (2, 1, 1), (3,1,1).
$$
\end{itemize}

\subsection{Type $(\pm\mp 3)$}
There are ten involutions such that $\mu$ is conjugate to $x\varpi_1-x\varpi_2+y\varpi_3$ for any $\lambda$, where $x\in\bbZ$ and $y\in\mathbb{P}$. We refer to these $s$-families as families of type $(\pm\mp 3)$. Whenever $x=0$, they will be of type $(3)$; whenever $x>0$, they will be of type $(23)$; and whenever $x<0$, they will be of type $(13)$.

\begin{center}
\begin{tabular}{l|r}
Type &  \#s \\
\hline
$(\pm\mp 3)$ & $39$, $41$, $58$, $68$, $74$, $80$, $87$, $95$, $98$, $113$
\end{tabular}
\end{center}

\subsection{Type $(\pm\mp 34)$}
There are five involutions such that $\mu$ is conjugate to $x\varpi_1-x\varpi_2+y\varpi_3+z\varpi_4$ for any $\lambda$, where $x\in\bbZ$ and $y, z\in\mathbb{P}$. We refer to these $s$-families as families of type $(\pm\mp 34)$.  Whenever $x=0$, they will be of type $(34)$; whenever $x>0$, they will be of type $(234)$; and whenever $x<0$, they will be of type $(134)$.

\begin{center}
\begin{tabular}{l|r}
Type &  \#s \\
\hline
$(\pm\mp 34)$ & $31$, $49$, $65$, $78$, $91$
\end{tabular}
\end{center}

\subsection{Type $(2\pm\mp)$}
There are eleven involutions such that $\mu$ is conjugate to $x\varpi_2+y\varpi_3-y\varpi_4$ for any $\lambda$, where $x\in\mathbb{P}$ and $y\in \bbZ$. We refer to these $s$-families as families of type $(2\pm\mp)$.  Whenever $y=0$, they will be of type $(2)$; whenever $y>0$, they will be of type $(24)$; and whenever $y<0$, they will be of type $(23)$.

\begin{center}
\begin{tabular}{l|r}
Type &  \#s \\
\hline
$(2\pm\mp)$ & $25$, $35$, $44$, $56$, $71$, $73$, $86$, $90$, $104$, $105$, $115$
\end{tabular}
\end{center}

\subsection{Type $(12\pm\mp)$}
There are four involutions such that $\mu$ is conjugate to $x\varpi_1+y\varpi_2+z\varpi_3-z\varpi_4$ for any $\lambda$, where $x, y\in\mathbb{P}$ and $z\in \bbZ$. We refer to these $s$-families as families of type $(12\pm\mp)$.  Whenever $z=0$, they will be of type $(12)$; whenever $z>0$, they will be of type $(124)$; and whenever $z<0$, they will be of type $(123)$.

\begin{center}
\begin{tabular}{l|r}
Type &  \#s \\
\hline
$(12\pm\mp)$ & $30$, $48$, $79$, $94$
\end{tabular}
\end{center}

\subsection{Remaining $s$-families}
There are twelve involutions whose types are no longer easily identified as above ones.
Their indices are
$$
81,  97, 100, 116, 117, 124, 126, 127, 128,
133, 134, 135.
$$
However, we can still handle them: there are just more cases. We illustrate the situation with an example.

\begin{example}
Let us consider the involution $s$ with index $100$. One calculates that there are five cases:
\begin{itemize}
  \item[$\bullet$]  $a>2d$, then $\mu=[a-2d,0,b+c+d,0]$ is of type $(13)$.
  \item[$\bullet$] $a=2d$, then $\mu=[0,0,b+c+d,0]$ is of type $(3)$.
  \item[$\bullet$] $d-b-c<a<2d$, then $\mu=[0,2d-a,a+b+c-d,0]$ is of type $(23)$.
  \item[$\bullet$] $a=d-b-c$,  then $\mu=[0,a+2b+2c,0, 0]$ is of type $(2)$.
  \item[$\bullet$] $a<d-b-c$, then $\mu=[0,a+2b+2c,0, d-a-b-c]$ is of type $(24)$.
\end{itemize}
Then for each case we can use the techniques from previous subsections. Finally, we know that there is no unitary representation in this $s$-family.
\qed
\end{example}

\section{Appendix}

In this appendix, we index all the involutions $s$ in the Weyl group of $F_4$ by presenting the weight $s\rho$.
\newpage
\begin{center}
\begin{tabular}{r|r|r|r|r|r}
Index &   $s\rho$  & Index &   $s\rho$ & Index &   $s\rho$\\ \hline
$1$ & $[1,1,1,1]$ & $2$  & $[-1,2,1,1]$ & $3$ & $[2,-1,2,1]$\\

$4$ & $[1,3,-1,2]$ & $5$  & $[1,1,2,-1]$ & $6$ & $[-1,4,-1,2]$\\

$7$ & $[-1,2,2,-1]$ & $8$  & $[2,-1,3,-1]$ & $9$ & $[-1,-1,3,1]$\\

$10$ & $[5,-3,1,3]$ & $11$  & $[4,1,-2,4]$ & $12$ & $[1,5,-1,-1]$\\

$13$ & $[-1,-1,4,-1]$ & $14$  & $[-1,6,-1,-1]$ & $15$ & $[3,-5,3,3]$\\

$16$ & $[5,-1,-1,4]$ & $17$  & $[4,5,-4,2]$ & $18$ & $[-5,1,1,4]$\\

$19$ & $[-3,5,-3,5]$ & $20$  & $[7,-5,4,-2]$ & $21$ & $[6,1,1,-4]$\\

$22$ & $[-4,-1,2,4]$ & $23$  & $[-5,3,-1,5]$ & $24$ & $[-3,9,-5,3]$\\

$25$ & $[5,-7,6,-2]$ & $26$  & $[10,-5,1,1]$ & $27$ & $[7,-1,2,-4]$\\

$28$ & $[1,1,-3,7]$ & $29$  & $[6,3,-1,-3]$ & $30$ & $[-7,1,5,-3]$\\

$31$ & $[-5,7,1,-5]$ & $32$  & $[1,-4,1,6]$ & $33$ & $[9,1,-3,1]$\\

$34$ & $[-1,-3,1,6]$ & $35$  & $[-6,-1,6,-3]$ & $36$ & $[-2,1,-2,7]$\\

$37$ & $[-11,5,1,1]$ & $38$  & $[-7,5,3,-5]$ & $39$ & $[-5,9,-1,-4]$\\

$40$ & $[1,-2,-1,7]$ & $41$  & $[7,-10,4,3]$ & $42$ & $[10,-1,-2,1]$\\

$43$ & $[10,-3,1,-2]$ & $44$  & $[5,5,-7,5]$ & $45$ & $[9,1,-2,-1]$\\

$46$ & $[1,1,4,-8]$ & $47$  & $[-1,-1,-1,7]$ & $48$ & $[-9,11,-4,1]$\\

$49$ & $[1,-6,8,-6]$ & $50$  & $[10,-1,-1,-1]$ & $51$ & $[-7,-4,5,2]$\\

$52$ & $[-1,-5,8,-6]$ & $53$  & $[-11,9,-2,1]$ & $54$ & $[-11,7,1,-2]$\\

$55$ & $[-2,1,5,-8]$ & $56$  & $[-6,11,-7,4]$ & $57$ & $[-9,11,-3,-1]$\\

$58$ & $[5,-10,8,-4]$ & $59$  & $[1,-2,6,-8]$ & $60$ & $[9,-4,-3,6]$\\

$61$ & $[3,7,-3,-5]$ & $62$  & $[-1,-1,6,-8]$ & $63$ & $[-11,9,-1,-1]$\\

$64$ & $[1,-11,7,1]$ & $65$  & $[1,10,-8,1]$ & $66$ & $[-1,-10,7,1]$\\

$67$ & $[-5,-6,8,-3]$ & $68$  & $[-9,6,-4,6]$ & $69$ & $[-1,11,-8,1]$\\

$70$ & $[-4,11,-4,-4]$ & $71$  & $[4,-11,4,4]$ & $72$ & $[1,-11,8,-1]$\\

$73$ & $[9,-6,4,-6]$ & $74$  & $[5,6,-8,3]$ & $75$ & $[1,10,-7,-1]$\\

$76$ & $[-1,-10,8,-1]$ & $77$  & $[-1,11,-7,-1]$ & $78$ & $[11,-9,1,1]$\\

$79$ & $[1,1,-6,8]$ & $80$  & $[-3,-7,3,5]$ & $81$ & $[-9,4,3,-6]$\\

$82$ & $[-1,2,-6,8]$ & $83$  & $[-5,10,-8,4]$ & $84$ & $[9,-11,3,1]$\\

$85$ & $[6,-11,7,-4]$ & $86$  & $[2,-1,-5,8]$ & $87$ & $[11,-7,-1,2]$\\

$88$ & $[11,-9,2,-1]$ & $89$  & $[1,5,-8,6]$ & $90$ & $[7,4,-5,-2]$\\

$91$ & $[-10,1,1,1]$ & $92$  & $[-1,6,-8,6]$ & $93$ & $[9,-11,4,-1]$\\

$94$ & $[1,1,1,-7]$ & $95$  & $[-1,-1,-4,8]$ & $96$ & $[-9,-1,2,1]$\\

$97$ & $[-5,-5,7,-5]$ & $98$  & $[-10,3,-1,2]$ & $99$ & $[-10,1,2,-1]$\\

$100$ & $[-7,10,-4,-3]$ & $101$  & $[-1,2,1,-7]$ & $102$ & $[5,-9,1,4]$\\

$103$ & $[7,-5,-3,5]$ & $104$  & $[11,-5,-1,-1]$ & $105$ & $[2,-1,2,-7]$\\

$106$ & $[6,1,-6,3]$ & $107$  & $[1,3,-1,-6]$ & $108$ & $[-9,-1,3,-1]$\\

$109$ & $[-1,4,-1,-6]$ & $110$  & $[5,-7,-1,5]$ & $111$ & $[7,-1,-5,3]$\\

$112$ & $[-6,-3,1,3]$ & $113$  & $[-1,-1,3,-7]$ & $114$ & $[-7,1,-2,4]$\\

$115$ & $[-10,5,-1,-1]$ & $116$  & $[-5,7,-6,2]$ & $117$ & $[3,-9,5,-3]$\\

$118$ & $[5,-3,1,-5]$ & $119$  & $[4,1,-2,-4]$ & $120$ & $[-6,-1,-1,4]$\\

$121$ & $[-7,5,-4,2]$ & $122$  & $[3,-5,3,-5]$ & $123$ & $[5,-1,-1,-4]$\\

$124$ & $[-4,-5,4,-2]$ & $125$  & $[-5,1,1,-4]$ & $126$ & $[-3,5,-3,-3]$\\

$127$ & $[1,-6,1,1]$ & $128$  & $[1,1,-4,1]$ & $129$ & $[-1,-5,1,1]$\\

$130$ & $[-4,-1,2,-4]$ & $131$  & $[-5,3,-1,-3]$ & $132$ & $[1,1,-3,-1]$\\

$133$ & $[-2,1,-3,1]$ & $134$  & $[1,-2,-2,1]$ & $135$ & $[1,-4,1,-2]$\\

$136$ & $[-1,-1,-2,1]$ & $137$  & $[-1,-3,1,-2]$ & $138$ & $[-2,1,-2,-1]$\\

$139$ & $[1, -2, -1, -1]$ & $140$  & $[-1, -1, -1, -1]$ &  &
\end{tabular}
\end{center}

\end{document}